\documentclass[12pt]{amsart}

\usepackage{a4wide}
\usepackage{color}
\usepackage{amsmath,amsfonts,amsthm,amssymb} 
\usepackage{graphicx}
\usepackage{comment}
\usepackage{hyperref}
\usepackage{cite}
\usepackage{accents}
\usepackage{enumerate}
\usepackage[us,12hr]{datetime}

\usepackage{longtable}

\newtheorem{theorem}{Theorem}
\newtheorem{lemma}[theorem]{Lemma}
\newtheorem{corollary}[theorem]{Corollary}
\newtheorem{proposition}[theorem]{Proposition}
\theoremstyle{definition}
\newtheorem{remark}[theorem]{Remark}

\newtheorem{definition}[theorem]{Definition}
\newtheorem{claim}[theorem]{Claim}

\numberwithin{equation}{section} \numberwithin{theorem}{section}

{\end{list}}

\def\XXint#1#2#3{{\setbox0=\hbox{$#1{#2#3}{\int}$}
     \vcenter{\hbox{$#2#3$}}\kern-.5\wd0}}

\newcommand{\mc}[1]{\mathcal{#1}}
\newcommand{\mbb}[1]{\mathbb{#1}}

\newcommand{\M}{\mathcal{M}}

\newcommand{\R}{\mbb{R}}

\newcommand{\ra}{\right\rangle}
\newcommand{\la}{\left\langle}
\newcommand{\e}{\varepsilon}

\newcommand{\pd}{\partial}
\newcommand{\cd}{\nabla}

\newcommand{\lb}{\left(}
\newcommand{\rb}{\right)}
\newcommand{\lsb}{\left[}
\newcommand{\rsb}{\right]}

\newcommand{\inner}[2]{\left\langle#1,#2\right\rangle} 

\newcommand{\ho}{\accentset{\circ}{A}}

\providecommand{\fff}{{g}}

\providecommand{\sff}{{\mathbf{A}}{}} 

\providecommand{\Wein}{{\mathbf{L}}{}}

\providecommand{\WeinN}{{L}}

\providecommand{\torsionhat}{\hat{\mathbf{T}}{}}
\providecommand{\torsionN}{{\mathrm{T}}{}}

\providecommand{\mn}{{\mathbf{H}}{}}

\providecommand{\mnN}{H} 

\providecommand{\nor}{{\mathbf{N}}}

\providecommand{\sffN}{{h}} 

\providecommand{\Ahat}{{\hat{\sff}}}

\providecommand{\NMN}{\nor M}
\providecommand{\fffN}{\fff^\nor} 
\providecommand{\cdN}{\cd^{\nor}} 

\providecommand{\NMhat}{\widehat{NM}}
\providecommand{\fffhat}{\hat\fff} 
\providecommand{\cdhat}{\hat\cd} 
\providecommand{\Deltahat}{\hat\Delta} 

\providecommand{\abs}[1]{\lvert#1\rvert}

\DeclareMathOperator{\spt}{spt}

\DeclareMathOperator{\Rm}{Rm}

\DeclareMathOperator{\tr}{tr}

\DeclareMathOperator{\dvg}{div}

\protected\def\vts{%
  \ifmmode
    \mskip0.5\thinmuskip
  \else
    \ifhmode
      \kern0.08334em
    \fi
  \fi
}

\def\labelitemi{--}
\def\ba #1\ea {\begin{align} #1\end{align}}
\def\bann #1\eann {\begin{align*} #1\end{align*}}
\def\ben #1\een {\begin{enumerate} #1\end{enumerate}}
\def\bi #1\ei {\begin{itemize}\renewcommand\labelitemi{--} #1\end{itemize}}

\address{Department of Mathematics, University of Tennessee Knoxville, Knoxville TN, USA, 37996-1320}
\address{School of Mathematical and Physical Sciences, The University of Newcastle, Newcastle, NSW, Australia, 2308}
\email{mlangford@utk.edu, mathew.langford@newcastle.edu.au}

\address{School of Mathematical Sciences, Queen Mary University of London, Mile End Road, London, UK, E1 4NS}
\email{h.nguyen@qmul.ac.uk}

\address{Eberhard Karls Universit\"at T\"ubingen, Fachbereich Mathematik, Auf der Morgenstelle 10, 72076 T\"{u}bingen, Germany}
\email{stephen.lynch@math.uni-tuebingen.de}

\begin{document}
\title[Quadratically pinched submanifolds]{Quadratically pinched submanifolds of the sphere via mean curvature flow with surgery}
\author{Mat Langford}
\author{Stephen Lynch}
\author{Huy The Nguyen}
\begin{abstract} 
We study mean curvature flow of $n$-dimensional submanifolds of $S_K^{n+\ell}$, the round $(n+\ell)$-sphere of sectional curvature $K>0$, under the quadratic curvature pinching condition $|\sff|^{2} < \frac{1}{n-2}\vert\mn\vert^{2} + 4K$ when $n\ge 8$, $|\sff|^{2} < \frac{4}{3n}\vert\mn\vert^{2}+\frac{n}{2}K$ when $n=7$, and $\vert \sff\vert^2<\frac{3(n+1)}{2n(n+2)}\vert\mn\vert^2+\frac{2n(n-1)}{3(n+1)}K$ when $n=5$ or $6$. This condition is related to a theorem of Li and Li [Arch. Math., 58:582--594, 1992] which states that the only $n$-dimensional minimal submanifolds of $S_K^{n+\ell}$ satisfying $\vert\sff\vert^2<\frac{2n}{3}K$ are the totally geodesic $n$-spheres. 
We prove the existence of a suitable mean curvature flow with surgeries starting from initial data satisfying the pinching condition. 
As a result, we conclude that any smoothly, properly immersed submanifold of $S_K^{n+1}$ satisfying the pinching condition is diffeomorphic either to the sphere $S^n$ or to the connected sum of a finite number of handles $S^1\times S^{n-1}$. The results are sharp when $n\ge 8$ due to hypersurface counterexamples.
\end{abstract}

\maketitle

\tableofcontents

\section{Introduction}


A famous theorem of Simons \cite{Si68} states that any minimal hypersurface of the round sphere $S^{n+1}$ with squared second fundamental form $\vert\sff\vert^2$ less than $n$ is necessarily a hyperequator. Simons' methods have been generalized in various directions \cite{AlencarDoCarmo,BakerNguyen17,ChengNakagawa,ChernDoCarmoKobayashi,LiLi,Okumura,Santos}, in particular to higher codimension minimal immersions. In this setting, the algebraic structure of Simons' equation becomes much more complicated, primarily due to the possibility of a non-trivially curved normal bundle. Nonetheless, Li--Li \cite{LiLi}, building on work of Chern--do Carmo--Kobayashi \cite{ChernDoCarmoKobayashi}, were able to show that a minimal immersion in $S^{n+\ell}$ satisfying $\vert\sff\vert^2<\frac{2n}{3}$ is necessarily totally geodesic. 

Such results can be improved upon using geometric flows \cite{An02,BakerThesis,Hu87,LangfordNguyen}. Building on work of Andrews--Baker \cite{AnBa10} and Baker \cite{BakerThesis}, Baker--Nguyen \cite{BakerNguyen2ElectricBoogaloo} showed that, under mean curvature flow, $n$-dimensional submanifolds of the sphere $S_K^{n+\ell}$ of sectional curvature\footnote{We find it convenient to work without normalizing the curvature $K$, as it serves as a natural scale parameter.} $K$ satisfying the quadratic curvature pinching condition
\begin{equation}\label{eq:AndrewsBaker pinching}
\vert\sff\vert^2<
\begin{cases}
\frac{1}{n-1}\vert\mn\vert^2+2K&\text{if}\;\; n\geq 4 \vts , \smallskip \\
\frac{4}{3n}\vert\mn\vert^2+\frac{n}{2}K&\text{if}\;\; n=2,3\vts,
\end{cases}
\end{equation}
where $\mn$ is the mean curvature vector, converge (preserving the inequality) either to a ``round'' point in finite time or to a totally geodesic subsphere in infinite time. In case $n\geq 4$, this behaviour is sharp in the sense that, for each $\varepsilon>0$, the Clifford embedding
\[
\mc M^{1,n-1}_\varepsilon\doteqdot \left\{(x,y)\in \R^2\times \R^n:\vert x\vert^2=\frac{1}{1+\varepsilon^2}\;\;\text{and}\;\;\vert y\vert^2=\frac{\varepsilon^2}{1+\varepsilon^2}\;\;\right\}
\]
of $S^1\times S^{n-1}$ into $S^{n+1}$ satisfies $\vert \sff\vert^2-\frac{1}{n-1}\vert\mn\vert^2-2=\frac{n-2}{n-1}\varepsilon^2$. 

Baker--Nguyen\cite{BakerNguyen2ElectricBoogaloo} (see also  \cite{BakerNguyen17}) refined this result in the context of surfaces in $S^4$ by including the curvature of the normal bundle in the pinching condition. Their pinching condition is less restrictive than \eqref{eq:AndrewsBaker pinching} (with $n=\ell=2$), and is almost sharp in that the Veronese surface, a minimal embedding of the projective plane into $S^4$, lies close to its boundary (they conjecture that the Veronese surface represents the sharp condition).

We will develop these results further by allowing a weaker curvature pinching condition. Namely, we study, for $n\ge 5$ and\footnote{Our arguments also apply in the codimension one case, $\ell=1$; however in this case the results obtained are weaker than known results \cite{LangfordNguyen}.} $\ell\ge 2$, $n$-dimensional submanifolds of $S_K^{n+\ell}$ satisfying the quadratic pinching condition
\begin{equation}\label{eq:strict quadratic pinching}
\vert\sff\vert^2<
\begin{cases}
\frac{1}{n-2}\vert\mn\vert^2+4K &\text{if}\;\; n\geq 8 \vts , \smallskip \\
\frac{4}{3n}\vert\mn\vert^2+\frac{n}{2}K&\text{if}\;\; n=7 \smallskip \\
\frac{3(n+1)}{2n(n+2)}\vert\mn\vert^2+\frac{2n(n-1)}{3(n+1)}K&\text{if}\;\; n=5,6\vts.
\end{cases}
\end{equation}
By constructing an appropriate mean curvature flow-with-surgeries, we are able to prove the following topological classification of submanifolds satisfying \eqref{eq:strict quadratic pinching}.

\begin{theorem}\label{thm:main theorem}
Every properly isometrically immersed $n$-dimensional submanifold $X:M\to S^{n+\ell}_K$ of $S^{n+\ell}_K$, $n\ge 5$, satisfying \eqref{eq:strict quadratic pinching} is diffeomorphic either to $S^n$ or to a connected sum of finitely many copies of $S^1\times S^{n-1}$. 
\end{theorem}


This theorem is sharp when $n\ge 8$ in the sense that, for each $\varepsilon>0$, the Clifford embedding
\[
\mc M^{1,n-2}_\varepsilon\doteqdot \left\{(x,y)\in \R^3\times \R^{n-1}:\vert x\vert^2=\frac{1}{1+\varepsilon^2}\;\;\text{and}\;\;\vert y\vert^2=\frac{\varepsilon^2}{1+\varepsilon^2}\;\;\right\}
\]
of $S^2\times S^{n-2}$ into $S^{n+1}$ satisfies $\vert \sff\vert^2-\frac{1}{n-2}\vert\mn\vert^2-4=2\frac{n-4}{n-2}\varepsilon^2$.

The first step in proving Theorem \ref{thm:main theorem} is to prove that the pinching condition is preserved. This is achieved by a fairly straightforward application of the maximum principle (cf. \cite[Lemma 3.1]{BakerNguyen2ElectricBoogaloo}). We then show that the conormal component of the second fundamental form is small compared to the mean curvature when the latter is large. 
Such a ``codimension estimate'' was first established by Naff \cite{Naff} for mean curvature flow in Euclidean spaces. His argument relies on the maximum principle but requires very careful accounting of first order terms. Our proof is inspired by Naff's but requires some effort to overcome the bad ambient curvature terms as well as the possible presence of points where the mean curvature is zero (at which the conormal subspace is not even defined). The codimension estimate is also a crucial ingredient in our ``cylindrical estimate'', in that it is needed to establish a suitable ``Poincar\'e-type inequality'' (cf. \cite[Proposition 3.2]{LyNgConvexity}), which is the key ingredient in a Huisken--Stampacchia iteration argument. The cylindrical estimate allows us to obtain pointwise estimates for the gradient and Hessian of the second fundamental form using the maximum principle. These estimates imply that the flow becomes either uniformly convex or quantitatively cylindrical with respect to a codimension one subspace in regions of high curvature. Once they are in place, we are able to apply the surgery apparatus developed by Huisken--Sinestrari \cite{HuSi09}, as extended by Nguyen \cite{Nguyen2020} to the high codimension setting. 

\subsection*{Acknowledgements}  
M.~Langford was supported by the Australian Research Council grant DE200101834. H.~T.~Nguyen was supported by the EPSRC grant EP/S012907/1.

\section{Preliminaries}

We begin by recalling the fundamental machinery regarding submanifolds of the sphere, and their evolution by mean curvature.

\subsection{Spaceforms}

Let $N_K=(N,g)$ be a complete Riemannian manifold of constant sectional curvature $K\in\R$, equipped with its Levi--Civita connection $D$. Recall that the curvature tensor $\Rm$ of $N$ is given by
\[
\Rm(u,v)w=K(g(u,w)v-g(v,w)u)\vts,
\]
where we use the convention
\[
\Rm^\cd(U,V)W\doteqdot \cd_{[U,V]}W-\cd_U(\cd_VW)+\cd_V(\cd_UW)
\]
for the curvature operator $\Rm^\cd$ of a connection $\cd$. 

\subsection{Submanifolds}

We follow \cite{AnBa10}. Consider an immersed submanifold $X:M^n\to N^{n+\ell}_K$ of a spaceform $N_K^{n+\ell}$. The \emph{normal bundle} $NM$ of $X$ is determined by the orthogonal decomposition
\[
X^\ast TN=dX(TM)\oplus^\perp NM\vts,
\]
where $X^\ast TN$ is the pullback of $TN$ and $dX:TM\to TN$ is the derivative of $X$. The pullback $X^\ast\!g$ of $g$ induces positive definite bilinear forms $\fff^\top$ on $TM$ (the \emph{first fundamental form}) and $\fff^\perp$ on $NM$, respectively. The pullback connection ${}^X\!D$ on $X^\ast TN$ induces connections $\cd^\top$ on $TM$  and $\cd^\perp$ on $NM$ via
\[
\cd^\top_uV\doteqdot ({}^X\!D_u[dX(V)])^\top \;\; \text{and}\;\; \cd^\perp_uN\doteqdot ({}^X\!D_uN)^\perp,
\]
where $\cdot^\top:X^\ast TN\to TM$ and $\cdot^\perp:X^\ast TN\to NM$ denote the tangential and normal (orthogonal) projections, respectively. The connection $\cd^\top$ induced on $TM$ is the Levi-Civita connection of $\fff^\top$ and $\cd^\perp$ is compatible with $\fff^\perp$.

The \emph{second fundamental form} $\sff\in\Gamma(T^\ast M\otimes T^\ast M\otimes NM)$ and the \emph{Weingarten tensor} $\Wein \in\Gamma(T^\ast M\otimes NM\otimes TM)$ are defined by the \emph{Gauss--Weingarten equations}
\begin{equation}\label{eq:first Weingarten general}
{}^X\!D_{u}[dX(V)]=dX(\cd^\top_uV)+\sff(u,v)
\end{equation}
and
\begin{equation}\label{eq:second Weingarten general}
{}^X\!D_vN
=\cd^\perp_vN-dX(\Wein(v,\nu))\vts,
\end{equation}
respectively, so that
\[
\fff^\perp(\sff(u,v),\nu)=\fff^\top(\Wein(u,\nu),v)\vts.
\]
These equations hold at any point $p\in M$ and for any tangent vectors $u,v\in T_pM$, normal vectors $\nu\in N_pM$, and extensions $V\in \Gamma(TM)$ of $v$ and $N\in\Gamma(NM)$ of $\nu$. 

\textbf{We will continue to use these conventions below; that is, given $p\in M$, the lower case latin letters $u,v,w,z$ will denote tangent vectors at $p$, the lower case greek letters $\mu,\nu$ will denote normal vectors at $p$, and corresponding upper case letters will denote corresponding extension fields, all of which may be arbitrarily chosen.}

Observe that
\begin{equation}\label{eq:Hessian of X is II}
(\cd_udX)(v)=\sff(u,v)\vts,
\end{equation}
where $\cd$ denotes the connection induced on the bundle $T^\ast M\otimes X^\ast TN$ by 
$\cd^\top$ and ${}^X\!D$. 
The covariant derivatives of the tangential and normal projections $\cdot^\top$ and $\cdot^\perp$ are given by
\begin{equation}\label{eq:derivative of top}
(\cd_u ( \vts \cdot^\top ) )(\phi)=\Wein(u,\phi^\perp)
\end{equation}
and
\begin{equation}\label{eq:derivative of perp}
(\cd_u ( \vts \cdot^\perp ) )(\phi)=-\sff(u,\phi^\top),
\end{equation}
for any $\phi\in X^\ast TN$, where $\cd$ denotes the connection induced on the relevant bundle by the connections ${}^X\!D$, $\cd^\top$ and $\cd^\perp$.

Differentiating \eqref{eq:first Weingarten general} and decomposing the result into tangential and normal components yields the \emph{Gauss equation},
\begin{equation}\label{eq:Gauss equation general}
K(\fff^\top(u,w)v-\fff^\top(v,w)u)=\Rm^\top(u,v)w+\Wein(v,\sff(u,w))-\Wein(u,\sff(v,w)),
\end{equation}
or, equivalently,
\ba\label{eq: Gauss arb codim}
K(\fff^\top(u,w)\fff^\top(v,z){}&-\fff^\top(u,z)\fff^\top(v,w))\nonumber\\
={}&\Rm^{\!\top\!}(u,v,w,z)+\fff^\perp\big(\sff(u,z),\sff(v,w)\big)-\fff^\perp\big(\sff(u,w),\sff(v,z)\big),
\ea
and the \emph{Codazzi--Mainardi equation},
\begin{equation}\label{eq:Codazzi equation general}
0=\nabla_v\sff(u,w)-\nabla_u\sff(v,w)\,,
\end{equation}
where, for any $\phi\in \Gamma(TN)$,
\[
{}^X\!\Rm(u,v)X^\ast\phi= \Rm(dX(u),dX(v))\phi\vts
\]
defines the curvature tensor ${}^X\!\Rm$ of ${}^X\!D$, $\Rm^\top$ denotes the curvature tensor of $\cd^\top$, and the covariant derivative of $\sff$ is defined in the canonical way:
\[
\cd_u\sff(v,w)\doteqdot \cd^\perp_u(\sff(V,W))-\sff(\cd^\top_uV,w)-\sff(v,\cd^\top_uW)\vts.
\]
The \emph{mean curvature}, $\mn$, is the trace of the second fundamental form,
\[
\mn\doteqdot \tr_{\fff^\top}\sff\,.
\]
The trace-free part is denoted by
\bann
\mathring{\sff}\doteqdot{}&\sff-\frac{1}{n}\mn\otimes \fff^\top\vts.
\eann
Decomposing $\cd\sff$ into its trace and trace-free parts yields the \emph{Kato inequality},
\begin{equation}\label{eq:Kato inequality}
\vert\cd\sff\vert^2\ge \frac{3}{n+2}\vert \cd^\perp \mn\vert^2.
\end{equation}

Differentiating \eqref{eq:second Weingarten general} and decomposing the result into normal and tangential components yields the \emph{Ricci equation},
\begin{equation}\label{eq:Ricci equation general}
0=\Rm^\perp(u,v)\nu+\sff(v,\Wein(u,\nu))-\sff(u,\Wein(v,\nu))\vts,
\end{equation}
and the Codazzi--Mainardi equation, respectively, where $\Rm^\perp$ denotes the curvature tensor of $\cd^\perp$. Contracting the Ricci equation yields the (\emph{contracted}) \emph{Ricci equation},
\bann
0={}&\Rm^\perp(u,v,\nu,\mu)+\fff^\top(\Wein(u,\mu),\Wein(v,\nu))-\fff^\top(\Wein(v,\mu),\Wein(u,\nu))\,.
\eann

Given a pair of tensors $S,T\in \Gamma(NM\otimes T^\ast M\otimes T^\ast M)$, we define a new tensor $S\wedge T\in \Gamma(NM\otimes NM\otimes T^\ast M\otimes T^\ast M)$ by
\begin{equation}\label{eq:wedge product}
(S\wedge T)(u,v)\doteqdot \tr_{\fff^\top}\big(S(u,\cdot)\otimes T(v,\cdot)-S(v,\cdot)\otimes T(u,\cdot)\big)\vts.
\end{equation}
With this notation, the Ricci equation becomes (after identifying $\Rm^\perp$ with a section of $T^\ast M\otimes T^\ast M\otimes NM\otimes NM$)
\begin{equation}\label{eq:Ricci equation wedge}
0=\Rm^{\perp\!}+\,\sff\wedge\sff\vts.
\end{equation}
Since
\[
\sff\wedge\sff=\mathring{\sff}\wedge\mathring{\sff}\vts,
\]
we may also write the Ricci equation as
\begin{equation}\label{eq:Ricci equation wedge trace free}
0=\Rm^{\perp\!}+\,\mathring{\sff}\wedge\mathring{\sff}\vts.
\end{equation}

Combining all of the above identities yields \emph{Simons' equation},
\ba\label{eq:Simons equation general}
\cd_u\cd_v\sff(w,z)-\cd_w\cd_z\sff(u,v)={}&\sff\big(u,\Wein\big(w,\sff(v,z)\big)\big)-\sff\big(w,\Wein\big(u,\sff(v,z)\big)\big)\nonumber\\
{}&+\sff\big(\Wein\big(w,\sff(u,v)\big),z\big)-\sff\big(\Wein\big(u,\sff(w,v)\big),z\big)\nonumber\\
{}&+\sff\big(v,\Wein\big(w,\sff(u,z)\big)\big)-\sff\big(v,\Wein\big(u,\sff(w,z)\big)\big)\nonumber\\
{}&+K\Big(\fff^\top(u,v)\sff\big(z,w\big)-\fff^\top(w,z)\sff\big(v,u\big)\nonumber\\
{}&\qquad\;+\fff^\top(u,z)\sff\big(v,w\big)-\fff^\top(w,v)\sff\big(z,u\big)\Big).
\ea
%
Taking the trace of \eqref{eq:Simons equation general} yields
\ba\label{eq:Simons equation trace}
\cd^\perp_u\cd^\perp_v\mn-\Delta\sff(u,v)={}&\tr_{\fff^\top}\Big(\sff\big(u,\Wein\big(\cdot,\sff(v,\cdot)\big)\big)+\sff\big(v,\Wein\big(\cdot,\sff(u,\cdot)\big)\big)\nonumber\\
{}&-2\sff\big(\cdot,\Wein\big(u,\sff(v,\cdot)\big)\big)+\sff\big(\cdot,\Wein\big(\cdot,\sff(u,v)\big)\big)\Big)\nonumber\\
{}&-\sff\big(v,\Wein\big(u,\mn\big)\big)+K\big(\fff^\top(u,v)\mn-n\sff\big(u,v\big)\big)\vts.
\ea

\subsection{Mean convex submanifolds}

If $\mnN \doteqdot \vert \mn\vert>0$ on $M$, then the mean curvature vector defines a canonical normal vector field $\nor\in\Gamma(NM)$, called the \emph{principal normal}, via
\[
\nor\doteqdot \frac{\mn}{\vert \mn\vert}.
\]
So the normal bundle splits globally as an orthogonal sum
\[
NM=\NMN\oplus^\perp \NMhat\vts,
\]
where the \emph{principal normal bundle} $\NMN\doteqdot \R\nor$ is the span of $\nor$ in $NM$, and the \emph{conormal bundle} $\NMhat$ is its orthogonal compliment in $NM$. 

The form $\fff^\perp$ induces positive definite forms $\fffN$ and $\fffhat$ on $\NMN$ and $\NMhat$, respectively, and the connection $\cd^\perp$ induces connections $\cdN$ and $\cdhat$ on $\NMN$ and $\NMhat$, respectively, in the usual way. 

Observe that $\cd^\perp_u\nor\in\NMhat_p$ for any $u\in T_pM$, since $\fff^\perp(\nor,\nor)\equiv 1$. Define the \emph{torsion} tensors $\torsionN \in\Gamma\big(T^\ast M\otimes \NMhat{}^\ast\big)$ 
and $\torsionhat\in\Gamma(T^\ast M \otimes \widehat{NM})$ 
by
\[
\torsionN(u,\mu)\nor\doteqdot (\cd^\perp_uM)^{\nor}\;\;\text{and}\;\; \torsionhat(u)\doteqdot (\cd^\perp_u\nor)^{\hat\perp}=\cd^\perp_u\nor\vts,
\]
where $\cdot^\nor$ and $\cdot^{\hat\perp}$ denote the projections onto $\NMN$ and $\NMhat$, respectively. 
Observe that
\[
\torsionN(u,\mu)+\fffhat(\torsionhat(u),\mu)=0\vts.
\]

We shall denote the principal normal component of $\sff$ by $\sffN$ and the conormal projection by $\Ahat$; that is,
 \[
\sffN\doteqdot \inner{\sff(u,v)}{\nor}\vts\;\;\text{and}\;\;\Ahat(u,v)\doteqdot \sff(u,v)-\inner{\sff(u,v)}{\nor}\nor\vts,
\]
so that
\ba\label{eq:sff decomposition}
\sff={}&\nor\otimes\sffN+\Ahat\vts.
\ea
With the notation of \eqref{eq:wedge product}, the Ricci equation \eqref{eq:Ricci equation general} then becomes
\ba
0
={}&\Rm^\perp+\vts(\nor\otimes\mathring\sffN+\Ahat)\wedge(\nor\otimes\mathring\sffN+\Ahat)\nonumber\\
={}&\Rm^\perp+\,(\nor\otimes\mathring\sffN)\wedge\Ahat+\Ahat\wedge(\nor\otimes\mathring\sffN)+\Ahat\wedge\Ahat\,,
\ea
where $\mathring\sffN\doteqdot \sffN-\frac{1}{n}H\fff^\top$ is the trace-free part of $\sffN$.
In particular,
\begin{equation}\label{eq:Ricci wedge trace free B}
\vert{\Rm^\perp}\vert^2=2\vert (\nor\otimes\mathring{\sffN})\wedge \Ahat\vert^2+\vert\Ahat\wedge \Ahat\vert^2\vts.
\end{equation}

Differentiating \eqref{eq:sff decomposition} yields
\bann
\cd_u\sff={}&\cd^\perp_u\nor\otimes \sffN+\nor\otimes\cd^\top_u\sffN+\cd_u\Ahat\\
={}&(\cd_u\Ahat)^\nor+\nor\otimes \cd^\top_u\sffN+(\cd_u\Ahat)^{\hat\perp}+
\torsionhat(u)\otimes \sffN\vts.
\eann
Since
\bann
(\cd_u\Ahat)^\nor(v,w)\doteqdot{}& \fff^\perp(\cd_u\Ahat(v,w),\nor)\nor\\
={}&\torsionN(u,\Ahat(v,w))\nor\vts,
\eann
and
\bann
(\cd_u\Ahat)^{\hat\perp}(v,w)\doteqdot{}& \cd_u\Ahat(v,w)-\fff^\perp(\cd_u\Ahat(v,w),\nor)\nor\\
={}&\cdhat_u\Ahat(v,w)\vts,
\eann
we obtain
\ba\label{eq:grad A decomposition}
\cd_u\sff={}&\nor\otimes\big(\torsionN(u,\Ahat)+\cd^\top_u\sffN\big)+\cdhat_u\Ahat+\torsionhat(u)\otimes \sffN\vts.
\ea

By the Codazzi--Mainardi equation, the tensors
\[
\torsionN(\cdot,\Ahat)+\cd^\top\sffN\;\;\text{and}\;\; \cdhat\Ahat+\torsionhat\otimes\sffN
\]
are totally symmetric in their $TM$ components. Decomposing them into their trace and trace-free components yields the Kato inequalities \cite{Naff}
\begin{equation}\label{eq:KatoA}
\vert\torsionN(\vts\cdot\vts,\Ahat)+\cd^\top\mathring{\sffN}\vert^2\ge \frac{2(n-1)}{n(n+2)}\vert\cd \mnN \vert^2
\end{equation}
and
\begin{equation}\label{eq:KatoB}
\vert\cdhat\Ahat+\torsionhat\otimes\mathring{\sffN}\vert^2\ge \frac{2(n-1)}{n(n+2)}\mnN^2\vert\torsionhat\vert^2\vts.
\end{equation}

We can play the same game with $\cd^2\sff$: differentiating \eqref{eq:grad A decomposition} yields
\bann
\cd_u\cd_v\sff&{}=\torsionhat(u)\otimes\big(\torsionN(v,\Ahat)+\cd^\top_v\sffN\big)+\nor\otimes\big(\cd_u\torsionN(v,\Ahat)+\torsionN(v,\cdhat_u\Ahat)+\cd^\top_u\cd^\top_v\sffN\big)\\
+{}&\nor\otimes\torsionN(u,\cdhat_v\Ahat)+\cdhat_u\cdhat_v\Ahat+\big(\nor\otimes\torsionN(u,\torsionhat(v))+\cd_u\torsionhat(v)\big)\otimes \sffN+\torsionhat(v)\otimes\cd^\top_u\sffN\vts,
\eann
where
\bann
\cd_u\torsionN(v,\mu)\doteqdot u(\torsionN(V,M))-\torsionN(\cd^\top_uV,\mu)-\torsionN(v,\cdhat_uM)
\eann
and
\bann
\cd_u\torsionhat(v)\doteqdot \cdhat_u(\torsionhat(V))-\torsionhat(\cd^\top_uV)\vts.
\eann
Taking the conormal projection yields
\bann
(\cd_u\cd_v\sff)^{\hat\perp}={}&\torsionhat(u)\otimes\big(\torsionN(v,\Ahat)+\cd^\top_v\sffN\big)+\cdhat_u\cdhat_v\Ahat+\cd_u\torsionhat(v)\otimes \sffN+\torsionhat(v)\otimes\cd^\top_u\sffN\vts,
\eann
and hence
\ba\label{eq:Laplace sff conormal}
(\Delta\sff)^{\hat\perp}={}&\tr_{\fff^\top}\big(\torsionhat(\vts\cdot\vts)\otimes\torsionN(\vts\cdot\vts,\Ahat)+2\torsionhat(\vts\cdot\vts)\otimes\cd^\top_{\cdot}\sffN\big)+\dvg\torsionhat\otimes \sffN+\Deltahat\Ahat\vts.
\ea

\subsection{Mean curvature flow}

Now consider a family $X:M^n\times I\to N_K^{n+\ell}$ of immersed submanifolds $X(\cdot,t):M^n\to N_K^{n+\ell}$ of $N_K^{n+\ell}$ which evolve by mean curvature flow. That is,
\[
\pd_tX=\mn\vts,
\]
where $\pd_tX\doteqdot dX(\pd_t)$, and $\pd_t$, the \emph{canonical vector field}, is defined via its action on functions $f\in C^\infty(M\times I)$ by
\[
\pd_t|_{(x_0,t_0)}f\doteqdot \frac{d}{dt}\Big|_{t=0}f(x_0,t_0+t)\vts.
\]

The tangent bundle to $M\times I$ splits as $T(M\times I)=TM\oplus \R\pd_t$, where we conflate $TM$ with the \emph{spatial tangent bundle}, $\{\xi\in T(M\times I):dt(\xi)=0\}$. Here, $dt$ is the one-form dual to $\pd_t$, or, equivalently, the differential of the \emph{time projection} $(p,t)\mapsto t$ from $M\times I$ to $I$.

We make use of the \emph{time-dependent connections} of Andrews--Baker \cite{AnBa10}, which extend the tangential and normal covariant derivatives to allow differentiation in space-time directions $\xi\in TM\oplus \R\pd_t$. That is,
\begin{equation*}
\nabla^\top_\xi V\doteqdot \big({}^X\!D_\xi \big[dX(V)\big]\big)^\top\;\;\text{and}\;\;\nabla^\perp_\xi N\doteqdot \big({}^X\!D_\xi N\big)^\perp,
\end{equation*}
where ${}^X\!D$ is the pullback of $D$ to $X^\ast TN$. Observe that $\cd^\top_\xi$ and $\cd_\xi^\perp$ coincide with $\cd^\top$ and $\cd^\perp$, respectively, when $\xi\in TM$, while
\ba\label{eq:spacetime connection}
\nabla^\top_t V=[\pd_t,V]-\Wein(V,\mn)\vts,
\ea
where $[\,\cdot\,,\cdot\,]$ denotes the Lie bracket and $\cd_t\doteqdot \cd_{\pd_t}$. 

The main advantage of working with the time-dependent connection (as opposed to the more commonly used Lie derivative) is that the tensors $\fff^\top$ and $\fff^\perp$ are $\pd_t$-parallel:
\[
\nabla^\top_{t}\fff^\top=0\;\;\text{and}\;\;\cd^\perp_t\fff^\perp=0\,.
\]
In order to exploit this, we first derive space-time analogues of the Codazzi--Gauss--Mainardi--Ricci--Weingarten equations (following \cite{AnBa10}). 

Observe that
\[
\sff(\pd_t,u)\doteqdot ({}^X\!D_t[dX(u)])^\perp=\cd^\perp_u\mn
\]
and hence
\[
\fff^\top(\Wein(\pd_t,\nu),u)=\fff^\perp(\cd^\perp_u\mn,\nu)\vts.
\]
Thus, proceeding as in the ``stationary'' case, we obtain the ``temporal'' Gauss equation
\begin{equation}\label{eq:temporal Gauss}
0=\Rm^\top(\pd_t,u,v,w)+\fff^\perp(\cd^\perp_w\mn,\sff(u,v))-\fff^\perp(\cd^\perp_v\mn,\sff(u,w))\vts,
\end{equation}
the temporal Codazzi--Mainardi equation
\begin{equation}\label{eq:temporal Codazzi}
\cd_t\sff(u,v)=\cd^\perp_u\cd^\perp_v\mn+\sff(\Wein(u,\mn),v)+\fff^\top(u,v)\mn,
\end{equation}
and the temporal Ricci equation
\begin{equation}\label{eq:temporal Ricci}
0=\Rm^\perp(\pd_t,u,\mu,\nu)+\fff^\perp(\cd^\perp_{\Wein(u,\nu)}\mn,\mu)-\fff^\perp(\cd^\perp_{\Wein(u,\mu)}\mn,\nu)\vts.
\end{equation}
The Gauss and Ricci equations allow us to interchange space-time covariant derivatives. The Codazzi equation provides an evolution equation for $\sff$.

Of course, the space-time connections $\cd^\top$ and $\cd^\perp$ also exhibit ``spatial'' Codazzi--Gauss--Mainardi--Ricci--Weingarten equations when restricted to the spatial tangent bundle.

Next, we derive a useful parabolic analogue of the Jacobi equation for minimal surfaces. It is not needed in the sequel, but it does provide some insight into the subsequent evolution equations.

\begin{proposition}[Jacobi equation]\label{prop:Jacobi}
Let $\{X_\varepsilon:M\times I\to N_K\}_{\varepsilon\in(-\varepsilon_0,\varepsilon_0)}$ be a $1$-parameter family of immersed mean curvature flows with $X_0=X$. The normal component $\sigma\doteqdot\left(\left.\frac{d}{d\varepsilon}\right|_{\varepsilon=0}X_\varepsilon\right)^\perp$ of the variation field satisfies the \textbf{Jacobi equation}
\bann
(\cd_t^\perp-\Delta^\perp)\sigma={}&\tr_{\fff^\top}\big(\sff(\,\cdot\,,\Wein(\,\cdot\,,\sigma))\big)+nK\sigma\vts.
\eann
\end{proposition}
\begin{proof}
Denote by $\tau$ the tangential component of the variation field. Fix $(p,t)$ and consider a normal coordinate neighbourhood $(U,\{x^i\}_{i=1}^n)$ for $M$ about $p$ with respect to $g_t$. Choosing $U$ smaller if necessary, we may identify $\tau$ with a section of $TU$. Using \eqref{eq:derivative of perp}, we compute at $(p,t)$
\bann
\cd^\perp_i\cd^\perp_j\sigma={}&\cd^\perp_i({}^X\!D_j\sigma)^\perp\\
={}&-\sff(\pd_i,({}^X\!D_j\sigma)^\top)+({}^X\!D_i({}^X\! D_j\sigma))^\perp\\
={}&\sff(\pd_i,\Wein(\pd_j,\sigma))+({}^X\!D_i({}^X\! D_j(\pd_\varepsilon X-\tau)))^\perp\\
={}&\sff(\pd_i,\Wein(\pd_j,\sigma))+\big({}^X\!D_i({}^X\! D_{\varepsilon}\pd_jX)-{}^X\!D_i(\cd^\top_j\tau+\sff(\pd_j,\tau))\big)^\perp\\
={}&\sff(\pd_i,\Wein(\pd_j,\sigma))+({}^X\!D_{\varepsilon}({}^X\!D_i\pd_jX)+{}^X\!\Rm(\pd_\varepsilon,\pd_i)\pd_j)^\perp\\
{}&-\sff(\pd_i,\cd_j\tau)-\cd_i\sff(\pd_j,\tau)-\sff(\pd_j,\cd^\top_i\tau)\\
={}&\sff(\pd_i,\Wein(\pd_j,\sigma))+({}^X\!D_{\varepsilon}(dX(\cd^\top_i\pd_j)+\sff_{ij})+\Rm(\sigma+\tau,\pd_iX)\pd_jX)^\perp\\
{}&-\sff(\pd_i,\cd_j\tau)-\cd_i\sff(\pd_j,\tau)-\sff(\pd_j,\cd^\top_i\tau)\vts.
\eann
Note that
\bann
{}^X\!D_{\varepsilon}[dX(\cd_i\pd_j)]={}&\pd_{\varepsilon}\Gamma_{ij}{}^kdX(\pd_k)+\Gamma_{ij}{}^k{}^X\!D_{\varepsilon}[dX(\pd_k)]\vts,
\eann
and hence
\[
({}^X\!D_{\varepsilon}[dX(\cd^\top_i\pd_j)])^\perp=0
\]
at $(p,t)$. Since, by the Codazzi equation,
\[
(\Rm(\tau,\pd_iX)\pd_jX)^\perp=0=\cd_i\sff(\tau,\pd_j)-\cd_\tau\sff(\pd_i,\pd_j),
\]
we obtain
\bann
\cd^\perp_i\cd^\perp_j\sigma={}&\sff(\pd_i,\Wein(\pd_j,\sigma))-(\Rm(\pd_iX,\sigma)\pd_jX)^\perp\\
{}&+({}^X\!D_{\varepsilon}(\sff_{ij}))^\perp-\sff(\pd_i,\cd^\top_j\tau)-\sff(\pd_j,\cd^\top_i\tau)-\cd_\tau\sff(\pd_j,\pd_j)\vts.
\eann

On the other hand, the mean curvature flow equation yields
\bann
{}^X\!D_t\sigma={}&{}^X\!D_t[dX(\pd_\varepsilon)-\tau]\\
={}&{}^X\!D_\varepsilon[dX(\pd_t)]-{}^X\!D_t\tau\\
={}&{}^X\!D_\varepsilon \mn-{}^X\!D_t\tau\vts.
\eann
Since
\bann
\pd_\varepsilon g_{ij}={}&\inner{{}^X\!D_\varepsilon\pd_iX}{\pd_jX}+\inner{\pd_iX}{{}^X\!D_\varepsilon\pd_jX}\\
={}&\inner{{}^X\!D_i(\tau+\sigma)}{\pd_jX}+\inner{\pd_iX}{{}^X\!D_j(\tau+\sigma)}\\
={}&g(\cd^\top_i\tau-\Wein(\pd_i,\sigma),\pd_j)+g(\pd_i,\cd^\top_j\tau-\Wein(\pd_j,\sigma))\vts,
\eann
and hence
\bann
\pd_\varepsilon g^{ij}={}&-g^{ip}g^{jq}\pd_\varepsilon g_{pq}\\
={}&-g^{ip}g^{jq}\left(g(\cd^\top_p\tau-\Wein(\pd_p,\sigma),\pd_q)\right)+g\left(\pd_p,\cd^\top_q\tau-\Wein(\pd_q,\sigma))\right)\vts,
\eann
%
%
%
we obtain
\bann
{}^X\!D_\varepsilon \mn={}&g^{ij}{}^X\!D_\varepsilon(\sff_{ij})-2g^{ip}g^{jq}g(\cd^\top_p\tau-\Wein(\pd_p,\sigma),\pd_q)\sff_{ij}\\
={}&g^{ij}\left({}^X\!D_\varepsilon(\sff_{ij})-2\sff(\cd^\top_i\tau-\Wein(\pd_i,\sigma),\pd_j)\right)
\eann
and hence
\[
\cd_t^\perp\sigma=\left(g^{ij}\left({}^X\!D_\varepsilon\sff_{ij}-2\sff(\cd^\top_i\tau-\Wein(\pd_i,\sigma),\pd_j)\right)-{}^X\!D_t\tau\right)^\perp\vts.
\]
Finally, since
\[
{}^X\!D_t\tau-{}^X\!D_\tau \mn=dX[\pd_t,\tau]=\frac{\pd \tau^i}{\pd t}dX(\pd_i)
\]
has no normal component, we conclude that
\bann
(\cd_t^\perp-\Delta^\perp)\sigma={}&\tr_{\fff^\top}\Big(\sff(\,\cdot\,,\Wein(\,\cdot\,,\sigma))+(\Rm(\,\cdot\,,\sigma)\,\cdot\,)^\perp\Big)\\
={}&\tr_{\fff^\top}\big(\sff(\,\cdot\,,\Wein(\,\cdot\,,\sigma))\big)+nK\sigma
\eann
as desired.
\end{proof}

Taking the trace of \eqref{eq:temporal Codazzi}, or applying Proposition \ref{prop:Jacobi} with the variation  $X_\varepsilon(x,t)\doteqdot X(x,t+\varepsilon)$, yields
\begin{equation}\label{eq:evolve H}
(\cd_t^\perp-\Delta^\perp)\mn=\tr_{\fff^\top}\big(\sff(\,\cdot\,,\Wein(\,\cdot\,,\mn))\big)+nK\mn\vts.
\end{equation}
It follows that
\begin{equation}\label{eq:evolve H squared}
(\pd_t-\Delta)\frac{1}{2}\vert \mn\vert^2=\vert\Wein(\,\cdot\,,\mn)\vert^2+nK\vert \mn\vert^2-\vert \cd^\perp \mn\vert^2\vts.
\end{equation}

Applying the trace Simons equation \eqref{eq:Simons equation trace} to the temporal Codazzi equation \eqref{eq:temporal Codazzi} yields
\ba\label{eq:evolve sff}
(\cd_t-\Delta)\sff(u,v)={}&\tr_{\fff^\top}\Big(\sff\big(u,\Wein\big(\cdot,\sff(v,\cdot\big)\big)+\sff\big(v,\Wein\big(\cdot,\sff(u,\cdot)\big)\big)-2\sff\big(\cdot,\Wein\big(u,\sff(v,\cdot)\big)\big)\nonumber\\
{}&+\sff\big(\cdot,\Wein\big(\cdot,\sff(u,v)\big)\big)\Big)+\fff^\top(u,v)K\mn-nK\Big(\sff\big(u,v\big)-\tfrac{1}{n}\fff^\top(u,v)\mn\Big).
\ea
Tracing this of course recovers \eqref{eq:evolve H}.

We also obtain
\begin{align}\label{eq:evolve A squared}
(\cd_t-\Delta)\frac{1}{2}\vert\sff\vert^2={}&-\vert\cd\sff\vert^2+\vert{\inner{\sff}{\sff}^\top}\vert^2+\vert{\mathring\sff\wedge\mathring\sff}\vert^2+nK\vert \sff\vert^2
-2nK\vert\mathring\sff\vert^2\vts,
\end{align}
where $\inner{\sff}{\sff}^\top\in \Gamma(NM\otimes NM)$ is formed from $\sff\otimes \sff$ by contracting the tangential components. That is, it is dual to the tensor $\inner{\Wein}{\Wein}^\top\in \Gamma(N^\ast M\otimes N^\ast M)$ defined by
\[
\inner{\Wein}{\Wein}^\top\!(\mu,\nu)\doteqdot\fff^\top\big(\Wein(\cdot,\mu),\Wein(\cdot,\nu)\big)\vts.
\]

Now consider points where $\mn\neq 0$. By \eqref{eq:evolve H squared},
\begin{equation}\label{eq:evolve H scalar}
(\pd_t-\Delta)\mnN =\big(\vert\sffN\vert^2-\vert\torsionhat\vert^2+nK\big)\mnN 
\end{equation}
and, since
\[
\Delta^\perp\nor=\dvg\torsionhat+\tr_{\fff^\top}\big(\torsionN(\vts\cdot\vts,\torsionhat(\vts\cdot\vts))\big)\nor\vts,
\]
\begin{equation}\label{eq:evolve nor}
\cd_t^\perp\nor=\dvg\torsionhat+\tr_{\fff^\top}\big(\torsionN(\vts\cdot\vts,\torsionhat(\vts\cdot\vts))\big)\nor+\tr_{\fff^\top}\big(\Ahat(\,\cdot\,,\Wein(\,\cdot\,,\nor))\big)+\vert\torsionhat\vert^2\nor+2\torsionhat\left(\cd\log\mnN\right).
\end{equation}

Projecting \eqref{eq:evolve sff} onto the conormal bundle yields
\bann
\big[(\cd_t-\Delta)\sff\big]^{\hat\perp}(u,v)={}&\tr_{\fff^\top}\Big(\Ahat\big(u,\Wein\big(\cdot,\sff(v,\cdot\big)\big)+\Ahat\big(v,\Wein\big(\cdot,\sff(u,\cdot)\big)\big)-2\Ahat\big(\cdot,\Wein\big(u,\sff(v,\cdot)\big)\big)\nonumber\\
{}&\qquad\;\;+\Ahat\big(\cdot,\Wein\big(\cdot,\sff(u,v)\big)\big)\Big)-nK\Ahat\big(u,v\big).
\eann

On the other hand, differentiating the principal-conormal decomposition of $\sff$ yields
\bann
\cd_t\sff={}&\cd_t^\perp\nor\otimes \sffN+\nor\otimes\cd_t^\top\sffN+\cd_t\Ahat\\
={}&\cd_t^\perp\nor\otimes \sffN+\cdhat_t\Ahat+\nor\otimes\big(\cd_t^\top\sffN+\torsionN(\vts\cdot\vts,\Ahat)\big)
\eann
and hence, recalling \eqref{eq:Laplace sff conormal} and \eqref{eq:evolve nor},
\bann
\big[(\cd_t-\Delta)\sff\big]^{\hat\perp}
={}&(\cdhat_t-\Deltahat)\Ahat+\big[\tr_{\fff^\top}\big(\Ahat(\,\cdot\,,\Wein(\,\cdot\,,\nor))\big)+2\torsionhat\left(\cd\log \mnN\right)\big]\otimes \sffN\\
{}&-\tr_{\fff^\top}\Big(\torsionhat(\vts\cdot\vts)\otimes\torsionN(\vts\cdot\vts,\Ahat)+2\torsionhat(\vts\cdot\vts)\otimes\cd_\cdot^\top\sffN\Big)\vts.
\eann
Since
\bann
\vert\cd\Ahat\vert^2={}&\vert\cdhat\Ahat\vert^2+\vert\torsionN(\cdot,\Ahat)\vert^2\\
={}&\vert\cdhat\Ahat\vert^2+\fffhat\big(\tr_{\fff^\top}(\torsionhat(\vts\cdot\vts)\otimes\torsionN(\vts\cdot\vts,\Ahat)),\Ahat\big)\vts,
\eann
we thus obtain, wherever $\mn\neq 0$ (cf. \cite{Naff}),
\ba\label{eq:evolve tor}
(\pd_t-\Delta)\frac{1}{2}\vert\Ahat\vert^2={}&\fffhat((\cdhat_t-\Deltahat)\Ahat,\Ahat)-\vert\cdhat\Ahat\vert^2\nonumber\\
={}&\big\vert\big\langle\sff,\Ahat\big\rangle^\top\big\vert^2-\big\vert\big\langle\nor\otimes\sffN,\Ahat\big\rangle^\top\big\vert^2+\vert\sff\wedge\Ahat\vert^2-nK\vert\Ahat\vert^2-\vert\cd\Ahat\vert^2\nonumber\\
{}&-2\mnN \fff\Big(\nor\otimes\cd\frac{\sffN}{\mnN },\cd\Ahat\Big)\nonumber\\
={}&\big\vert\big\langle\Ahat,\Ahat\big\rangle^{\!\top}\big\vert^2+\vert\Ahat\wedge\Ahat\vert^2+\vert\nor\otimes\sffN\wedge \Ahat\vert^2-nK\vert\Ahat\vert^2-\vert\cd\Ahat\vert^2\nonumber\\
{}&-2\mnN \fff\Big(\nor\otimes\cd\frac{\sffN}{\mnN },\cd\Ahat\Big)\vts.
\ea

Given tensor fields $S$ and $T$ formed from $TM$ and $NM$, we denote by $S\ast T$ any tensor field resulting from linear combinations of contractions of $S\otimes T$ with $\fff^\top$ and $\fff^\perp$. By \eqref{eq:spacetime connection}, \eqref{eq:temporal Gauss}, and \eqref{eq:temporal Ricci},
\bann
\cd_t(\cd_{\cdot\,} T)
={}&\cd(\cd_tT)+\sff\ast \sff\ast \cd T+\sff\ast\cd \sff\ast T.
\eann
By \eqref{eq:Gauss equation general} and \eqref{eq:Ricci equation general},
\bann
\Delta(\cd T)
={}&\cd(\Delta T)+\sff\ast \sff\ast \cd T+K\ast \cd T+\sff\ast \cd \sff\ast T\,.
\eann
Thus,
\bann
(\cd_t-\Delta )(\cd \sff)
={}&\sff\ast \sff\ast \cd \sff+K\ast\cd \sff\,,
\eann
and hence, by Young's inequality,
\begin{align}
(\partial_t-\Delta)|\nabla \sff|^2\le{}& -2 |\nabla^2 \sff|^2 + c(|\sff|^2+K)|\nabla \sff|^2\,,\label{eqn_evolderiv}
\end{align}
where $c$ is a constant that depends only on $n$ and $k$.

Similarly,
\bann
(\cd_t-\Delta )(\cd^2 \sff)={}&\sff\ast \sff\ast \cd^2 \sff+\sff\ast\cd \sff\ast \cd \sff+K\ast\cd^2 \sff\,,
\eann
and hence
\begin{align}
(\partial_t-\Delta)|\nabla^2\sff|^2\le{}&-2|\nabla^3\sff|^2 + c\lsb (|\sff|^2+K)|\nabla^2\sff|^2+|\cd \sff|^4\rsb,\label{eqn_evolderiv2}
\end{align}
where $c$ is a constant that depends only on $n$ and $k$.

Similar inequalities hold for higher derivatives of $\sff$ since, by an induction argument,
\ba\label{eq:evolve derivatives of A}
(\cd_t-\Delta )(\cd^m \sff)={}&K\ast\cd^m\sff+\sum_{i+j+k=m}\cd^i\sff\ast \cd^j\sff\ast \cd^k\sff\,.
\ea

The following ``Bernstein estimates'' are a standard application of the ``rough'' evolution equations \eqref{eq:evolve derivatives of A}. For a proof in the Euclidean case (which carries over with minor modifications) see, for example, {\cite[Theorem 6.24]{EGF}}.

\begin{proposition}[Bernstein estimates]\label{prop:Bernstein}
Let $X:M\times [0,\lambda K^{-1}]\to N_K^{n+\ell}$ be a solution to mean curvature flow, where $K>0$. If
\[
\max_{M\times[0,\lambda K^{-1}]}\vert \sff\vert^2\le \Lambda_0K\,,
\]
then
\[
t^m\vert\cd^m \sff\vert^2\le \Lambda_{m}K\,,
\]
where $\Lambda_{m}$ depends only on $n$, $k$, $m$, $\lambda$ and $\Lambda_0$.
\end{proposition}

\subsection{A Poincar\'e-type inequality}

We need the following Poincar\'e-type inequality, which combines the arguments of \cite[Proposition 2.2]{LangfordNguyen} and \cite[Proposition 3.2]{LyNgConvexity}. 
Note that the hypothesis \eqref{eq:codimension hypothesis} is motivated by the codimension estimate proved in Section \ref{sec:codimension} below.

\begin{proposition}\label{prop:Poincare}
There exists $\gamma=\gamma(n,\alpha,\eta,\delta,\Lambda)>0$ with the following property. Given a smoothly immersed submanifold $X: M^n\to S^{n+\ell}_K$ satisfying
\begin{equation}\label{eq:codimension hypothesis}
\vert\Ahat\vert^2\le \Lambda K^\delta\big(\vert\mn\vert^2+K\big)^{1-\delta}\;\;\text{wherever}\;\;\mn\neq 0\vts,
\end{equation}
define the ``acylindrical'' set $U_{\alpha,\eta}\subset M$ by
\[
U_{\alpha,\eta}\doteqdot \left\{x\in M:|\sff|^2-\tfrac{1}{n-2+\alpha}\vert\mn\vert^2-2(2-\alpha)K\leq 0\leq |\sff|^2-\left(\tfrac{1}{n-1}+\eta\right)\vert\mn\vert^2\right\}\vts.
\]
If $u\in W^{1,2}(M)$ satisfies $\operatorname{spt}u\Subset U_{\alpha,\eta}$, then
\[
\gamma\int u^2(\vert\mn\vert^2+K)\vts d\mu\leq \int u^2\left(K+\frac{|\nabla u|}{u}\frac{|\nabla \sff|}{\sqrt{\vert\mn\vert^2+K}}+\frac{|\nabla \sff|^2}{\vert\mn\vert^2+K}\right)d\mu\vts.
\]
\end{proposition}
\begin{proof}
Since $C^\infty(M)$  is dense in $W^{1,2}(M)$, it suffices to establish the estimate for smooth $u$. Define a tensor $E\in\Gamma(T^\ast M\otimes T^\ast M\otimes T^\ast M\otimes T^\ast M)$ by
\bann
E(u,v,w,z)\doteqdot{}&\cd_{(u}\cd_{v)}\sff(w,z)-\cd_{(w}\cd_{z)}\sff(u,v)\\
\doteqdot{}&\frac{1}{2}\big(\cd_{u}\cd_{v}\sff(w,z)-\cd_w\cd_z\sff(u,v)+\cd_{v}\cd_{u}\sff(w,z)-\cd_z\cd_v\sff(u,v)\big)\vts.
\eann
We first consider points where $\mn\neq 0$. By Simons' equation \eqref{eq:Simons equation general}, arguing as in \cite[Lemma 3.1]{LyNgConvexity} yields a constant $C$ depending only on $n$ such that
\begin{equation}\label{eq: E to F}
\vert E\vert^2\ge 2\big\vert\sffN\otimes\sffN^2-\sffN^2\otimes\sffN+K(g\otimes\sffN-\sffN\otimes g)\big\vert^2-C(\vert\sffN\vert^5+K^{\frac{5}{2}})\vert\Ahat\vert
\end{equation}
wherever $\vert\mn\vert>0$. 

Let us define 
\[F\doteqdot\sffN\otimes\sffN^2-\sffN^2\otimes\sffN+K(g\otimes\sffN-\sffN\otimes g)\vts.\]
We claim there is a constant $\gamma=\gamma(n,\alpha,\eta)$ such that 
\begin{equation}\label{eq:poinc_F_lower}
\vert F\vert^2+\vert\mn\vert^5\vert\Ahat\vert+K^3\geq\gamma W^3
\end{equation}
at points in $U_{\alpha,\eta}$ with $\mn\neq0$, where
\begin{equation*}
W\doteqdot \vert\mn\vert^2+K\vts.
\end{equation*}
In fact, we will show that \eqref{eq:poinc_F_lower} holds as an algebraic inequality. If we write $\lambda_p$ for the eigenvalues of $h$, this inequality can be rewritten as
\[\sum_{p,q}(\lambda_p\lambda_q+K)^2(\lambda_p-\lambda_q)^2+\tr(\lambda)^5\vert\Ahat\vert+K^3\geq\gamma (\tr(\lambda)^2 + K)^3,\]
where
\begin{equation}\label{eq:Poincare constants}
|\lambda|^2\doteqdot \sum_p \lambda_p^2\;\;\text{and}\;\; \tr(\lambda)\doteqdot\sum_p\lambda_p\vts.
\end{equation}
Thus, if \eqref{eq:poinc_F_lower} is false (as an algebraic inequality), then we can find a sequence of vector-valued symmetric bilinear forms $\sff_i \in \mathbb{R}^n\odot\mathbb{R}^n\otimes\mathbb{R}^\ell$ with nonvanishing trace such that the conditions $\tr(\lambda^i)>0$, 
\begin{equation}
\label{eq:poinc_sequence_1}
0 \leq \vert\lambda^i\vert^2+\vert\Ahat_i\vert^2-\left(\tfrac{1}{n-1}+\eta\right)\tr(\lambda^i)^2,
\end{equation}
and
\begin{equation}
\label{eq:poinc_sequence_2}
\vert\lambda^i\vert^2+\vert\Ahat_i\vert^2-\tfrac{1}{n-2+\alpha}\tr(\lambda^i)^2-2(2-\alpha)K\leq 0
\end{equation}
hold for every $i\in\mathbb{N}$, and yet
\begin{equation}
\label{eq:poinc_sequence_3}
\frac{\sum_{p,q}(\lambda_p^i\lambda_q^i+K)^2(\lambda_p^i-\lambda_q^i)^2+\tr(\lambda^i)^5\vert\Ahat_i\vert+K^3}{(\tr(\lambda^i)^2 + K)^3}\to0
\end{equation}
as $i \to \infty$. It follows that $\tr(\lambda^i)^2 \to \infty$ as $i \to \infty$, and hence, as a consequence of \eqref{eq:poinc_sequence_2},
\[\tr(\lambda^i)^{-2}\vert\lambda^i\vert^2+\tr(\lambda^i)^{-2}\vert\Ahat_i\vert^2\leq\tfrac{1}{n-2+\alpha}+2(2-\alpha)K\tr(\lambda^i)^{-2}\vts.\]
So (passing to a subseqence if necessary) we may extract limits
\[\tr(\lambda^i)^{-1}\lambda^i \to \lambda^\infty \;\;\text{and}\;\; \tr(\lambda^i)^{-1}\Ahat_i \to \Ahat_\infty.\]
Passing to the limit in \eqref{eq:poinc_sequence_3} we find that
\[\sum_{p,q}(\lambda_p^\infty\lambda_q^\infty)^2(\lambda_p^\infty-\lambda_q^\infty)^2+\vert\Ahat_\infty\vert=0,\]
and hence $\Ahat_\infty=0$ and $\lambda^\infty = \tfrac{1}{m}\sum_{p=1}^m e_p$ for some $1\leq m \leq n$, where $\{e_p\}_{p=1}^n$ is the standard basis for $\mathbb{R}^n$. Passing to limits in \eqref{eq:poinc_sequence_1} and \eqref{eq:poinc_sequence_2}, and inserting $\Ahat_\infty=0$ and $|\lambda^\infty|^2 = m^{-1}$, we obtain 
\[ \frac{1}{n-1}+\eta \leq \frac{1}{m}\leq\frac{1}{n-2+\alpha}.\]
It follows that $n-2<m <n-1$, which is absurd. We conclude that our initial assumption was false; that is, \eqref{eq:poinc_F_lower} holds in $U_{\alpha,\eta}$ at points where $\mn\neq 0$. Recalling \eqref{eq: E to F} then yields
\[
\gamma W^3
\le K^3+\vert E\vert^2+CW^{\frac{5}{2}}\vert\Ahat\vert\vts,
\]
where $C=C(n,\alpha)$. Applying the hypothesis \eqref{eq:codimension hypothesis} and Young's inequality then yields
\[
\gamma W^3\le K^3+\vert E\vert^2
\]
with $\gamma$ taking a smaller value and depending now also on $\delta$ and $\Lambda$.

On the other hand, wherever $\mn=0$,
\[
W^3=K^3\le \vert E\vert^2+K^3\vts.
\]
We conclude that
\[
\gamma W^3\le \vert E\vert^2+K^3
\]
at all points of $U_{\alpha,\eta}$, where $\gamma=\gamma(n,\alpha,\eta,\delta,\Lambda)$. Thus,
\bann
\gamma\!\int\! u^2Wd\mu\leq{}& \int  \frac{u^2}{W^2}\left(\vert E\vert^2+K^3\right)d\mu\\
\leq{}& \int  u^2\left(\frac{\vert E\vert^2}{W^2}+K\right)d\mu\vts.
\eann
We now estimate, using Simons' equation and the divergence theorem,
\bann
\int \frac{u^2}{W^2}\vert E\vert^2d\mu={}& \int\frac{u^2}{W^{2}}E\ast\nabla^2\sff \,d\mu\\
={}&\int\frac{u^2}{W^{2}}\!\left(\frac{\nabla u}{u}\ast E+\frac{\nabla W}{W}\ast E+\nabla E\right)\!\ast\nabla\sff\,d\mu\\
\leq{}&C\int \frac{u^2}{W^{2}}\left(W^{\frac{3}{2}}\frac{|\nabla u|}{u}+W^{\frac{1}{2}}|\nabla W|+W|\nabla\sff|\right)|\nabla\sff|\,d\mu\\
\leq{}&C\int u^2\left(\frac{|\nabla u|}{u}+\frac{|\nabla\sff|}{W^{\frac{1}{2}}}\right)\frac{|\nabla\sff|}{W^{\frac{1}{2}}}\,d\mu\vts,
\eann
where $C$ depends only on $n$, $\alpha$ and $\eta$. This completes the proof.
\end{proof}





\subsection{Notation}\mbox{}

Before moving on, let us review our notation for the tensors which will appear frequently in the sequel.\bigskip

\begin{longtable}[l]{ll}
$\nor$ & Principal normal vector\\ 
$\sff$ & Second fundamental form\\ 
$\sffN$ & Principal normal component of $\sff$\\ 
$\Ahat$ & Conormal projection of $\sff$\\
$\Wein$ & Weingarten tensor\\ 
$\mn$ & Mean curvature vector\\ 
$\mnN$ & Mean curvature scalar\\
$\mathring\sff$ & Trace-free part of $\sff$\\
$\mathring\sffN$ & Trace-free part of $\sffN$\\ 
$\torsionN$ & Torsion (scalar valued)\\
$\torsionhat$ & Torsion (conormal valued)\\
\end{longtable}

\section{The key estimates for smooth flows}

In this section, we obtain the key estimates in the setting of smooth flows. In the following section, we will obtain extensions of appropriately modified versions of these estimates for surgically modified flows.

\subsection{Preserving quadratic pinching} \mbox{}

If $X_0:M^n\to S^{n+\ell}_K$ satisfies \eqref{eq:strict quadratic pinching}, then we can find $\alpha\in(\alpha_n,1)$ such that
\begin{equation}\label{eq:uniform quadratic pinching}
\vert\sff\vert^2-\frac{1}{n-2+\alpha}\vert\mn\vert^2-2(2-\alpha)K\le 0\vts,
\end{equation}
where
\begin{equation}\label{eq:alpha_n}
\alpha_n\doteqdot \max\{2-\tfrac{n}{4},0\}=\begin{cases} \frac{3}{4} & \text{if}\;\; n=5\\ \frac{1}{2} & \text{if}\;\; n=6\\ \frac{1}{4}  & \text{if}\;\; n=7\\  0 & \text{if}\;\; n\ge 8\vts.\end{cases}
\end{equation}
We will show that this is preserved under mean curvature flow. In fact, we will prove (more generally) that the condition
\begin{equation}\label{eq:uniform quadratic pinching general}
\vert\sff\vert^2-\frac{1}{n-m+\alpha}\vert\mn\vert^2-2(m-\alpha)K\le 0
\end{equation}
is preserved for any integer $m\ge 1$ and any $\alpha\in[0,1)$ so long as $m-\alpha\le \frac{n}{4}$.

\begin{proposition}[Quadratic pinching is preserved]\label{prop:preserving pinching}
Let $X:M^n\times[0,T)\to S^{n+\ell}_K$ be a solution to mean curvature flow such that \eqref{eq:uniform quadratic pinching} holds on $M^n\times\{0\}$ for some integer $m\ge 1$ and some $\alpha\in[0,1)$. If $m-\alpha\le \frac{n}{4}$, then \eqref{eq:uniform quadratic pinching} holds on $M^n\times\{t\}$ for all $t\in[0,T)$.
\end{proposition}
\begin{proof}
Given positive numbers $a$ and $b$, consider the function
\[
Q\doteqdot \frac{1}{2}\left(\vert \sff\vert^2-a\vert \mn \vert^2-bK\right)\,.
\]
By \eqref{eq:evolve A squared} and \eqref{eq:evolve H squared},
\bann
(\pd_t-\Delta)Q={}&\vert{\mathring\sff\wedge\mathring\sff}\vert^2+\vert{\inner{\sff}{\sff}^\top}\vert^2-a\vert\Wein(\,\cdot\,,\mn)\vert^2+nK\big(\vert \sff\vert^2-a\vert \mn\vert^2\big)-2nK\vert\mathring\sff\vert^2\\
{}&-\big(\vert\cd\sff\vert^2-a\vert \cd^\perp \mn\vert^2\big)\vts.
\eann

At a point where $\mn\neq 0$, decomposing $\sff$ into its irreducible components yields \cite{AnBa10}
\[
\vert\sff\vert^2
=\vert\mathring{\sffN}\vert^2+\tfrac{1}{n}\mnN^2+\vert\Ahat\vert^2\vts,
\]
\[
\vert\Wein(\cdot,\mn)\vert^2
=\vert\mathring{\sffN}\vert^2\mnN^2+\tfrac{1}{n}\mnN^4\vts,
\]
\[
\vert{\mathring\sff\wedge\mathring\sff}\vert^2=2\vert (\nor\otimes\mathring{\sffN})\wedge \Ahat\vert^2+\vert\Ahat\wedge \Ahat\vert^2\vts,
\]
and
\bann
\vert{\inner{\sff}{\sff}^\top}\vert^2={}&\big\vert\big\langle \nor\otimes\mathring{\sffN}+\tfrac{1}{n}\mnN \nor g^\top+\Ahat,\nor\otimes\mathring{\sffN}+\tfrac{1}{n}\mnN \nor g^\top+\Ahat\big\rangle\big\vert\\
={}&\big(\vert\mathring{\sffN}\vert^2+\tfrac{1}{n}\mnN^2\big)^2+2\big\vert\big\langle\nor\otimes\mathring{\sffN},\Ahat\big\rangle^\top\big\vert^2+\vert{\big\langle\Ahat,\Ahat\big\rangle^{\!\top}}\vert^2\vts.
\eann
Applying the estimates \cite[p. 372]{AnBa10}
\begin{equation}\label{eq:AnBa reaction 1}
\vert \nor\otimes\mathring{\sffN}\wedge \Ahat\vert^2+\big\vert\big\langle\mathring{\sffN},\Ahat\big\rangle^\top\big\vert^2\le 2\vert \mathring{\sffN}\vert^2\vert\Ahat\vert^2
\end{equation}
and \cite[Proposition 3]{LiLi} (cf. \cite[Lemma 1]{ChernDoCarmoKobayashi})
\begin{equation}\label{eq:AnBa reaction 2}
\vert\Ahat\wedge\Ahat\vert^2+\vert{\big\langle\Ahat,\Ahat\big\rangle^{\!\top}}\vert^2\le\tfrac{3}{2}\vert\Ahat\vert^4
\end{equation}
now yields
\bann
\vert{\mathring\sff\wedge\mathring\sff}\vert^2+\vert{\inner{\sff}{\sff}^\top}\vert^2
\le{}&4\vert \mathring{\sffN}\vert^2\vert\Ahat\vert^2+\tfrac{3}{2}\vert\Ahat\vert^4+\big(\vert\mathring{\sffN}\vert^2+\tfrac{1}{n}\mnN^2\big)^2\\
={}&3\vert \mathring{\sffN}\vert^2\vert\Ahat\vert^2+\tfrac{3}{2}\vert\Ahat\vert^4+\big(\vert\mathring{\sffN}\vert^2+\tfrac{1}{n}\mnN^2\big)\vert\sff\vert^2-\tfrac{1}{n}\vert\Ahat\vert^2\mnN^2
\eann
and hence
\bann
(\pd_t-\Delta)Q={}&\vert{\mathring\sff\wedge\mathring\sff}\vert^2+\vert{\inner{\sff}{\sff}^\top}\vert^2-a\big(\vert\mathring{\sffN}\vert^2+\tfrac{1}{n}\mnN^2\big)\mnN^2+nK\big(2Q+bK\big)-2nK\vert\mathring\sff\vert^2\\
{}&-\big(\vert\cd\sff\vert^2-a\vert \cd^\perp \mn\vert^2\big)\\
\le{}&3\vert \mathring{\sffN}\vert^2\vert\Ahat\vert^2+\tfrac{3}{2}\vert\Ahat\vert^4-\tfrac{1}{n}\vert\Ahat\vert^2\mnN^2-2nK\big(\vert\mathring\sffN\vert^2+\vert\Ahat\vert^2\big)\\
{}&+\big(\vert\mathring{\sffN}\vert^2+\tfrac{1}{n}\mnN^2+nK\big)\big(2Q+bK\big)-\big(\vert\cd\sff\vert^2-a\vert \cd^\perp \mn\vert^2\big)\\
={}&\big(3\vert \mathring{\sffN}\vert^2+\tfrac{3}{2}\vert\Ahat\vert^2-\tfrac{1}{n}\mnN^2-bK\big)\vert\Ahat\vert^2\\
{}&+bK\big(\vert\mathring{\sffN}\vert^2+\vert\Ahat\vert^2+\tfrac{1}{n}\mnN^2+nK\big)-2nK\big(\vert\mathring\sffN\vert^2+\vert\Ahat\vert^2\big)\\
{}&+2Q\big(\vert\mathring{\sffN}\vert^2+\tfrac{1}{n}\mnN^2+nK\big)-\big(\vert\cd\sff\vert^2-a\vert \cd^\perp \mn\vert^2\big)\vts.
\eann

Rewriting
\[
\vert\sff\vert^2 = 2Q+a\mnN^2+bK\vts,
\]
we obtain 
\bann
bK\big(\vert\mathring{\sffN}\vert^2+\vert\Ahat\vert^2+\tfrac{1}{n}\mnN^2+nK\big)={}& bK(2Q+a\mnN^2+(b+n)K)
\eann
and
\bann
-2nK\big(\vert\mathring\sffN\vert^2+\vert\Ahat\vert^2\big)={}&-2nK\big(2Q+\big(a-\tfrac{1}{n}\big)\mnN^2+bK\big)\vts,
\eann
and hence
\bann
bK\big(\vert\mathring{\sffN}\vert^2+\vert\Ahat\vert^2+\tfrac{1}{n}\mnN^2+nK\big)-{}&2nK\big(\vert\mathring\sffN\vert^2+\vert\Ahat\vert^2\big)\\
={}&-2K(2n-b)Q+(a(b-2n)+2)K\mnN^2+b(b-n)K^2.
\eann

Similarly, using 
\[
\tfrac{1}{n}\mnN^2=\tfrac{1}{an-1}\big(\vert\mathring{\sffN}\vert^2+\vert\Ahat\vert^2-2Q-bK\big)
\]
we find 
\bann
3\vert \mathring{\sffN}\vert^2+\tfrac{3}{2}\vert\Ahat\vert^2-\tfrac{1}{n}\mnN^2-bK={}&3\vert \mathring{\sffN}\vert^2+\tfrac{3}{2}\vert\Ahat\vert^2-\tfrac{1}{an-1}\big(\vert\mathring{\sffN}\vert^2+\vert\Ahat\vert^2-2Q-bK\big)-bK\\
={}&\big(3-\tfrac{1}{an-1}\big)\vert \mathring{\sffN}\vert^2+\big(\tfrac{3}{2}-\tfrac{1}{an-1}\big)\vert\Ahat\vert^2+\tfrac{2}{an-1}Q-\big(1-\tfrac{1}{an-1}\big)bK\vts.
\eann
If $a\leq\tfrac{4}{3n}$, then the first term on the right is nonpositive. Discarding this term and putting things back together, we arrive at
\bann
(\pd_t-\Delta)Q\leq{}&\big(\tfrac{3}{2}-\tfrac{1}{an-1}\big)\vert\Ahat\vert^4-\big(1-\tfrac{1}{an-1}\big)bK\vert\Ahat\vert^2+b(b-n)K^2\\
&+(a(b-2n)+2)K\mnN^2+2Q\big(\vert\mathring{\sffN}\vert^2+\tfrac{1}{an-1}\vert\Ahat\vert^2+\tfrac{1}{n}\mnN^2+(b-n)K\big)\\
&-\big(\vert\cd\sff\vert^2-a\vert \cd^\perp \mn\vert^2\big)\vts.
\eann

Observe that if
\[
a=\frac{1}{n-m+\alpha}\;\;\text{and}\;\; b=2(m-\alpha)
\]
for some $m$  and $\alpha\in(0,1)$ such that $m-\alpha\le \frac{n}{4}$ (which ensures that  $a\leq\tfrac{4}{3n}$), then
\[
a(b-2n)+2 = 0\;\; \text{and}\;\; \big(1-\tfrac{1}{an-1}\big)b=2(b-n).
\]
We arrive at
\bann
(\tfrac{3}{2}-\tfrac{1}{an-1}\big)\vert\Ahat\vert^4-\big(1-\tfrac{1}{an-1}\big)bK\vert\Ahat\vert^2+b(b-n)K^2={}&(\tfrac{3}{2}-\tfrac{1}{an-1}\big)\vert\Ahat\vert^4-2(b-n)K\vert\Ahat\vert^2\\
&+b(b-n)K^2. 
\eann
The discriminant of the quadratic form on the right is $3(b-n)(\tfrac{n}{3}-m+\alpha)$. Since $m-\alpha \leq \tfrac{n}{4}$ and $b=2(m-\alpha)$, this is strictly negative. 
Thus,
\ba\label{eq:Q reaction term estimate}
(\pd_t-\Delta)Q\leq{}&2Q\big(\vert\mathring{\sffN}\vert^2+\tfrac{1}{an-1}\vert\Ahat\vert^2+\tfrac{1}{n}\mnN^2+(b-n)K\big)-\big(\vert\cd\sff\vert^2-a\vert \cd^\perp \mn\vert^2\big)\vts.
\ea
Since $n\ge 1+\frac{2(m-\alpha)}{3}$, the gradient terms can be discarded using the Kato inequality.

On the other hand, wherever $\mn =0$,
\bann
(\pd_t-\Delta)Q={}&\vert\mathring{\sff}\wedge \mathring{\sff}\vert^2+\big\vert\big\langle\mathring\sff,\mathring\sff\big\rangle^\top\big\vert^2-nK\vert\mathring\sff\vert^2-\big(\vert\cd\sff\vert^2-a\vert \cd^\perp \mn\vert^2\big)\vts.
\eann
So \cite[Proposition 3]{LiLi} yields
\bann
(\pd_t-\Delta)Q\le{}&\big(\tfrac{3}{2}\vert\mathring{\sff}\vert^2-nK\big)\vert\mathring\sff\vert^2-\big(\vert\cd\sff\vert^2-a\vert \cd^\perp \mn\vert^2\big)\\
={}&3\vert\mathring{\sff}\vert^2Q+\big(\tfrac{3}{2}b-n\big)K\vert\mathring\sff\vert^2-\big(\vert\cd\sff\vert^2-a\vert \cd^\perp \mn\vert^2\big)\vts.
\eann
Since $b=2(m-\alpha)<\frac{n}{2}$, the Kato inequality then yields
\bann
(\pd_t-\Delta)Q\le{}&3\vert\mathring{\sff}\vert^2Q\vts.
\eann
We conclude that
\bann
(\pd_t-\Delta)Q\le{}&fQ
\eann
everywhere for some locally bounded function $f$, at which point we may conclude that non-positivity of $Q$ is preserved.
\end{proof}

We can use the preservation of pinching to obtain an estimate for the trace-free second fundamental form which improves at large times, using the maximum principle.
\begin{proposition}\label{prop:cylindrical decay}
Let $X:M\times[0,T)\to S_K^{n+\ell}$ be a solution to mean curvature flow with initial condition in satisfying \eqref{eq:uniform quadratic pinching general} with $m-\alpha<\frac{n}{4}$. There is a constant $C=C(n,m-\alpha)$ such that 
\begin{equation}
\frac{\vert\sff\vert^2-\frac{1}{n}\vert\mn\vert^2}{\vert\mn\vert^2+K}\le C\mathrm{e}^{-2Kt}\vts.
\end{equation}
\end{proposition}
\begin{proof}
Set
\[
g \doteqdot \frac{1}{2}\big(\vert \sff\vert^2-\tfrac{1}{n}\vert\mn\vert^2\big)\;\;\text{and}\;\;W\doteqdot \frac{1}{2}\big(bK+a\vert\mn\vert^2-\vert\sff\vert^2\big)\vts,
\]
where
\[
a=\frac{4}{3n}\;\;\text{and}\;\; b=\frac{n}{2}.
\]

Applying \eqref{eq:evolve H squared} and \eqref{eq:evolve A squared} yields (cf. Proposition \ref{prop:preserving pinching}), wherever $\mn\neq 0$,
\bann
(\pd_t-\Delta)g
\le{}&\big(3\vert \mathring{\sffN}\vert^2+\tfrac{3}{2}\vert\Ahat\vert^2-\tfrac{1}{n}\mnN^2\big)\vert\Ahat\vert^2-2nK\big(\vert\mathring\sffN\vert^2+\vert\Ahat\vert^2\big)\\
{}&+2g\big(\vert\mathring{\sffN}\vert^2+\tfrac{1}{n}\mnN^2+nK\big)-\big(\vert\cd\sff\vert^2-\tfrac{1}{n}\vert \cd^\perp \mn\vert^2\big)\vts.
\eann
Note that
\[
\vert\mathring\sffN\vert^2+\vert\Ahat\vert^2=2g
\]
so that
\[
3\vert \mathring{\sffN}\vert^2+\tfrac{3}{2}\vert\Ahat\vert^2-\tfrac{1}{n}\mnN^2=(3-\tfrac{1}{an-1})\vert \mathring{\sffN}\vert^2+(\tfrac{3}{2}-\tfrac{1}{an-1})\vert\Ahat\vert^2-\tfrac{1}{n}\mnN^2+\tfrac{2}{an-1}g\vts.
\]
The first three terms are non-positive. The remaining term is also non-positive by the Kato inequality. Thus,
\ba
(\pd_t-\Delta)g\le{}&-4nKg+2g\big(\vert\mathring{\sffN}\vert^2+\tfrac{1}{an-1}\vert\Ahat\vert^2+\tfrac{1}{n}\mnN^2+nK\big)\nonumber\\
={}&2g\big(\vert\mathring{\sffN}\vert^2+\tfrac{1}{an-1}\vert\Ahat\vert^2+\tfrac{1}{n}\mnN^2-nK\big)\vts.\label{eq:evolve g eta H nonzero}
\ea

Since
\bann
-(\pd_t-\Delta)W\le-2W\big(\vert\mathring{\sffN}\vert^2+\tfrac{1}{an-1}\vert\Ahat\vert^2+\tfrac{1}{n}\mnN^2-(n-b)K\big)-2\gamma\vert\cd\sff\vert^2\vts,
\eann
where
\[
2\gamma=1-\tfrac{n+2}{3}a>0\vts,
\]
we obtain
\bann
\frac{(\pd_t-\Delta)\frac{g}{W}}{\frac{g}{W}}={}&\frac{(\pd_t-\Delta)g}{g}-\frac{(\pd_t-\Delta)W}{W}+2\inner{\cd\log\frac{g}{W}}{\cd\log W}\\
\le{}&-2bK-2\gamma\frac{\vert\cd\sff\vert^2}{W}+2\inner{\cd\log\frac{g}{W}}{\cd\log W}
\eann
wherever $\mn\neq 0$.

On the other hand, wherever $\mn =0$,
\bann
-(\pd_t-\Delta)W\le{}&-3\vert\mathring{\sff}\vert^2W
\eann
and
\ba\label{eq:evolve g eta H zero}
(\pd_t-\Delta)g\le{}&\big(3\vert\mathring{\sff}\vert^2-2K\big)g\vts,
\ea
and hence, at such points,
\bann
\frac{(\pd_t-\Delta)\frac{g}{W}}{\frac{g}{W}}\le{}&-2K+2\inner{\cd\log\frac{g}{W}}{\cd\log W}.
\eann

Since $b \geq 1$ we conclude that
\bann
\frac{(\pd_t-\Delta)\frac{g}{W}}{\frac{g}{W}}\le{}&-2K+2\inner{\cd\log\frac{g}{W}}{\cd\log W}
\eann
everywhere. The maximum principle now implies the claim.
\end{proof}

Proposition \ref {prop:cylindrical decay} implies, in particular, that there exists $\tau = \tau(n,m-\alpha)$ such that the inequality 
\[\vert\sff\vert^2 - \tfrac{1}{n-1} \mnN^2 - 2K < 0\]
holds for each $t \in(0,T)\cap[\tau K^{-1},\infty)$ on any solution initially satisfying \eqref{eq:uniform quadratic pinching general}. If $T > \tau K^{-1}$, this means that at time $\tau K^{-1}$ the solution satisfies the hypotheses of \cite[Main Theorem 7]{BakerThesis}. Consequently, the solution either exists forever and converges to a totally geodesic submanifold as $t \to \infty$, or else contracts to a round point in finite time. 

We also find that the only quadratically pinched ancient solutions are the totally umbilic ones (cf. \cite{ChernDoCarmoKobayashi,LaNgPinching,LiLi,LyNg}).

\begin{theorem}
Let $X:M\times(-\infty,0)\to S_K^{n+\ell}$ be a proper ancient solution to mean curvature flow. If
\[
\limsup_{t\to-\infty}\left(\vert\sff\vert^2-\tfrac{4}{3n}\vert\mn\vert^2-\tfrac{n}{2}K\right)<0\vts,
\]
then $X(M,t)$ is totally umbilic for each $t$, and hence, up to a rotation, the solution is either a stationary hyperequator or a shrinking hyperparallel in $S_K^{n+1}\hookrightarrow S_K^{n+\ell}$.
\end{theorem}

\subsection{The codimension estimate}\label{sec:codimension}

Next, we obtain an estimate for $\vert\Ahat\vert$ which improves at high curvature scales for submanifolds satisfying \eqref{eq:strict quadratic pinching}. Such an estimate was obtained for high codimension mean curvature flow in Euclidean space by Naff \cite{Naff}. 
Two new difficulties need to be overcome in our setting, however: the first is the fact that our pinching condition allows the mean curvature vector to vanish at some points (at which $\Ahat$ is not defined); the second is the presence of ambient curvature terms, which need to be controlled.

In fact, the estimate holds under the condition \eqref{eq:uniform quadratic pinching general}, so long as
\[
m-\alpha<\min\left\{\frac{n}{4},\frac{n(n-1)}{3(n+1)}\right\},
\]
which when $m=2$ is implied by \eqref{eq:strict quadratic pinching}. Note that this corresponds exactly to the constants in Naff's estimate in the Euclidean setting \cite{Naff}.

\begin{proposition}[Codimension estimate (cf. \cite{Naff})]\label{prop:codimension estimate}
Let $X:M^n\times[0,T)\to S^{n+\ell}_K$ be a solution to mean curvature flow such that \eqref{eq:uniform quadratic pinching general} holds on $M^n\times\{0\}$ for some $m\ge 1$ and $\alpha\in[0,1)$. If $m-\alpha<\min\left\{\frac{n}{4},\frac{n(n-1)}{3(n+1)}\right\}$, then

\[
\vert\Ahat\vert^2\le CK^\delta(\vert \mn\vert^2+K)^{1-\delta}\vts \;\;\text{wherever}\;\; \mn\neq 0\vts,
\] 
where $\delta>0$ depends only on $n$ and $m-\alpha$, and $C<\infty$ depends only on $n$, $m-\alpha$, and
an upper bound for $K^{-1}\max_{M\times\{0\}}\vert\sff\vert^2$.
\end{proposition}
\begin{proof}
Let $\tau = \tau(n,m-\alpha)$ be defined by
\[2 \tau\doteqdot \min\left\{\frac{n}{4},\frac{n(n-1)}{3(n+1)}\right\} - m + \alpha.\]
Given $\sigma\in(0,1)$, set
\[
f_\sigma\doteqdot \begin{cases} \displaystyle \frac{\frac{1}{2}\vert\Ahat\vert^2}{W}W^{\sigma} &\text{if} \;\; H >0\\ 0 & \text{if}\;\; H=0\vts,\end{cases}
\]
where 
\[
W\doteqdot \frac{1}{2}\big(bK+a\mnN^2-\vert\sff\vert^2\big)
\]
with
\[
a=\frac{1}{n-m+\alpha -\tau }\;\;\text{and}\;\; b=2(m - \alpha + \tau)\vts.
\]
Observe that
\[
2W\ge \varepsilon b(\mnN^2+K)\vts,
\]
where $\varepsilon=\varepsilon(n,m-\alpha)>0$. Indeed, since $(m-\alpha)$-pinching is preserved when $m-\alpha<\frac{n}{4}$,
\bann
-2W={}&-\left(\frac{1}{n-m+\alpha -\tau } - \frac{1}{n-m+\alpha} \right)\mnN^2-2\tau K+\vert\sff\vert^2-\tfrac{1}{n-m+\alpha}\mnN^2-2(m-\alpha)K\\
\le{}&-\left(\frac{1}{n-m+\alpha-\tau} - \frac{1}{n-m+\alpha} \right)\mnN^2-2\tau K.
\eann
Moreover, with this choice of $a$ we have 
\begin{equation}
\label{eq:codim_a}
a < \min\left\{\frac{4}{3n}, \frac{3(n+1)}{2n(n+2)}\right\}.
\end{equation}

Recall that
\begin{align*}
(\cd_t-\Delta)\frac{1}{2}\vert\Ahat\vert^2={}&\vert{\big\langle\Ahat,\Ahat\big\rangle^{\!\top}}\vert^2+\vert \sffN\wedge \Ahat\vert^2+\vert\Ahat\wedge\Ahat\vert^2-nK\vert\Ahat\vert^2\\
{}&-\vert\cd\Ahat\vert^2-2\mnN \fff\Big(\nor\otimes\cd\frac{\sffN}{\mnN},\cd\Ahat\Big)
\end{align*}
and

\begin{align*}
-(\pd_t-\Delta)W={}&\vert{\inner{\sff}{\sff}^\top}\vert^2+\vert{\sff\wedge\sff}\vert^2-a\vert\Wein(\,\cdot\,,\mn)\vert^2+nK\big(\vert \sff\vert^2-a\vert \mn\vert^2\big)-2nK\vert\mathring\sff\vert^2\\
{}&-\big(\vert\cd\sff\vert^2-a\vert \cd^\perp \mn\vert^2\big)
\end{align*}
wherever $\mn\neq 0$. We first compare the reaction terms in these two evolution equations. 
\begin{claim}\label{claim:Naff reaction}
There exists $\theta=\theta(n,m-\alpha)<1$ such that
\ba\label{eq:Naff reaction}
\frac{1}{\frac{1}{2}\vert\Ahat\vert^2}\Big({}&\vert{\big\langle\Ahat,\Ahat\big\rangle^{\!\top}}\vert^2+\vert \sffN\wedge \Ahat\vert^2+\vert\Ahat\wedge\Ahat\vert^2-nK\vert\Ahat\vert^2\Big)\nonumber\\
{}&\le-\frac{\theta}{W}\Big(\vert{\inner{\sff}{\sff}^\top}\vert^2+\vert{\sff\wedge\sff}\vert^2-a\vert\Wein(\,\cdot\,,\mn)\vert^2+nK\big(\vert \sff\vert^2-a\vert \mn\vert^2\big)-2nK\vert\mathring\sff\vert^2\!\Big)\vts.
\ea
\end{claim}
\begin{proof}[Proof of Claim \ref{claim:Naff reaction}]
On the one hand, we have seen (recall \eqref{eq:Q reaction term estimate}) that
\bann
-R_{W}\doteqdot{}&\vert{\inner{\sff}{\sff}^\top}\vert^2+\vert{\sff\wedge\sff}\vert^2-a\vert\Wein(\,\cdot\,,\mn)\vert^2+nK\big(\vert \sff\vert^2-a\vert \mn\vert^2\big)-2nK\vert\mathring\sff\vert^2\\
\le{}&-2W\big(\vert\mathring{\sffN}\vert^2+\tfrac{1}{an-1}\vert\Ahat\vert^2+\tfrac{1}{n}\mnN^2-(n-b)K\big)\vts.
\eann
%
Replacing $bK$ using
\[
\big(1+\tfrac{\varepsilon}{2}\big)bK=\big(1+\tfrac{\varepsilon}{2}\big)\big(\vert\mathring\sffN\vert^2+\vert\Ahat\vert^2+2W-\tfrac{an-1}{n}\mnN^2\big),
\]
we obtain
\bann
R_W\ge{}&2W\Big(\big(2+\tfrac{\varepsilon}{2}\big)\vert\mathring{\sffN}\vert^2+\big(1+\tfrac{\varepsilon}{2}+\tfrac{1}{an-1}\big)\vert\Ahat\vert^2+2\big(1+\tfrac{\varepsilon}{2}\big)W\Big)\\
&+2W\Big(\left(\tfrac{2-an}{n}-\tfrac{\varepsilon}{2}\tfrac{an-1}{n} \right) \mnN^2-\big(n+\tfrac{\varepsilon b}{2}\big)K\!\Big).
\eann
The term $\frac{2-an}{n}\mnN^2$ can be discarded since $2-an>0$.

Since $2W\ge \varepsilon b(H^2+K)$, we may estimate
\[
-4W+\varepsilon bK\le -\varepsilon b\big(2\mnN^2+K\big),
\]
and hence
\bann
R_W\ge{}&W\Big(\big(4+\varepsilon\big)\vert\mathring{\sffN}\vert^2+\big(2+\varepsilon+\tfrac{2}{an-1}\big)\vert\Ahat\vert^2+(\varepsilon b-2n)K+2\varepsilon W+\varepsilon(2b-\tfrac{an-1}{n})H^2\Big)\\
\ge{}&W\Big(\big(4+\varepsilon\big)\vert\mathring{\sffN}\vert^2+\big(3+\varepsilon\big)\vert\Ahat\vert^2+(\varepsilon-2n)K\Big)\vts.
\eann

On the other hand, by \eqref{eq:AnBa reaction 1} and \eqref{eq:AnBa reaction 2},
\bann
R_{\frac{1}{2}\vert\Ahat\vert^2}\doteqdot{}&\vert{\big\langle\Ahat,\Ahat\big\rangle^{\!\top}}\vert^2+\vert \sffN\wedge \Ahat\vert^2+\vert\Ahat\wedge\Ahat\vert^2-nK\vert\Ahat\vert^2\\
\le{}& \big(4\vert\mathring\sffN\vert^2+3\vert\Ahat\vert^2-2nK\big)\tfrac{1}{2}\vert\Ahat\vert^2.
\eann
The claim follows.
\end{proof}

We estimate the first order terms as follows.

\begin{claim}\label{claim:codimension estimate gradient terms}
There exist $\theta=\theta(n,m-\alpha)<1$ and $\Lambda=\Lambda(n,m-\alpha)<\infty$ such that
\begin{equation}\label{eq:Naff gradient}
-\vert\cd\Ahat\vert^2-2\mnN \fff\Big(\nor\otimes\cd\frac{\sffN}{\mnN},\cd\Ahat\Big)\le \theta\frac{\frac{1}{2}\vert\Ahat\vert^2}{W}\big(\vert\cd\sff\vert^2-a\vert\cd^\perp\mn\vert^2\big)
\end{equation}
so long as $H^2\ge \Lambda K$.
\end{claim}
\begin{proof}[Proof of Claim \ref{claim:codimension estimate gradient terms}]
Decomposing $\cd\sff$ and $\cd\mn$ into their irreducible components and applying the Kato inequalities \eqref{eq:KatoA} and \eqref{eq:KatoB} yields
\bann
\vert\cd\sff\vert^2-a\vert\cd\mn\vert^2={}&-\tfrac{an-1}{n}\big(\mnN^2\vert\torsionhat\vert^2+\vert\cd\mnN\vert^2\big)+\vert \torsionhat\otimes\mathring\sffN+\cdhat\Ahat\vert^2+\vert\torsionN(\vts\cdot\vts,\Ahat)+\cd^\top\mathring\sffN\vert^2\\
\ge{}&\big(\tfrac{2(n-1)}{n(n+2)}(1-s_2)-\tfrac{an-1}{n}\big)\mnN^2\vert\torsionhat\vert^2+\big(\tfrac{2(n-1)}{n(n+2)}(1-s_1)-\tfrac{an-1}{n}\big)\vert\cd\mnN\vert^2\\
{}&+s_1\vert\torsionN(\vts\cdot\vts,\Ahat)+\cd^\top\mathring\sffN\vert^2+s_2\vert\torsionhat\otimes\mathring\sffN+\cdhat\Ahat\vert^2
\eann
for any $s_1,s_2\in[0,1]$. Choosing $s_2=0$ and $s_1= 1-\frac{(n+2)(an-1)}{2(n-1)}$, so that
\[
\tfrac{2(n-1)}{n(n+2)}(1-s_1)-\tfrac{an-1}{n}=0\vts,
\]
yields
\ba\label{eq:good grad terms}
\vert\cd\sff\vert^2-a\vert\cd\mn\vert^2\ge{}&\alpha_1\vert\torsionN(\vts\cdot\vts,\Ahat)+\cd^\top\mathring\sffN\vert^2+\frac{\alpha_2}{n}\mnN^2\vert\torsionhat\vert^2,
\ea
where
\[
\alpha_1\doteqdot 1-\tfrac{(n+2)(an-1)}{2(n-1)}\;\;\text{and}\;\; \alpha_2\doteqdot \tfrac{2(n-1)}{n+2}-(an-1).
\]

The two components of $\cd\sff$ on the right of \eqref{eq:good grad terms} will be sufficient to control the gradient terms in the evolution equation for $\Ahat$ (which is why we kept as much of them as possible, at the expense of the others).

So consider
\bann
-2\mnN \fff\Big(\nor\otimes\cd\frac{\sffN}{\mnN},\cd\Ahat\Big)={}&2\mnN\tr_{\fff^\top}\fff\Big(\cd_{\cdot}^\perp\nor\otimes\cd_\cdot\frac{\mathring\sffN}{\mnN},\Ahat\Big)\\
\leq{}&2\vert\torsionhat\vert\vert\Ahat\vert\left\vert \cd\mathring\sffN-\frac{\cd H}{H}\otimes\mathring\sffN\right\vert\\
={}&2\vert\torsionhat\vert\vert\Ahat\vert\left\vert\cd\mathring\sffN+\torsionN(\vts\cdot\vts,\Ahat)-\torsionN(\vts\cdot\vts,\Ahat)-\frac{\cd H}{H}\otimes\mathring\sffN\right\vert\\
\le{}&2\vert\torsionhat\vert\vert\Ahat\vert\left(\vert\cd\mathring\sffN+\torsionN(\vts\cdot\vts,\Ahat)\vert+\vert\torsionN(\vts\cdot\vts,\Ahat)\vert+\vert\cd H\vert\frac{\vert\mathring\sffN\vert}{H}\right).
\eann
Let $\mathring\chi $ be the indicator function on the support of $\vert\mathring\sffN\vert^2$. Then, applying Young's inequality three times yields, for any $\gamma_1,\gamma_2>0$,
\bann
-2\mnN \fff\Big(\nor\otimes\cd\frac{\sffN}{\mnN},\cd\Ahat\Big)\le{}&(1+\gamma_1^{-1}+\gamma_2^{-1}\mathring\chi)\vert\torsionhat\vert^2\vert\Ahat\vert^2\\
{}&+\gamma_1\vert\cd\mathring\sffN+\torsionN(\vts\cdot\vts,\Ahat)\vert^2+\vert\torsionN(\vts\cdot\vts,\Ahat)\vert^2+\gamma_2 \mathring\chi\frac{\vert\mathring\sffN\vert^2}{H^2}\vert\cd H\vert^2.
\eann
On the other hand, by Young's inequality and the Kato inequality \eqref{eq:KatoB},
\bann
\vert\cdhat\Ahat\vert^2+\vert\mathring\sffN\vert^2\vert\torsionhat\vert^2\ge{}&\tfrac{1}{2}\vert\cdhat\Ahat+\torsionhat\otimes\mathring\sffN\vert^2\\
\ge
{}&\frac{n-1}{n(n+2)}\mnN^2\vert\torsionhat\vert^2.
\eann 
Combining these yields
\bann
-\vert\cd\Ahat\vert^2-2\mnN \fff\Big(\nor\otimes\cd\frac{\sffN}{\mnN},\cd\Ahat\Big)\le{}&\left(\!(1+\gamma_1^{-1}+\gamma_2^{-1}\mathring\chi)\frac{\vert\Ahat\vert^2}{H^2}+\frac{\vert\mathring\sffN\vert^2}{H^2}-\frac{n-1}{n(n+2)}\right)\!H^2\vert\torsionhat\vert^2\\
{}&+\gamma_1\vert\cd\mathring\sffN+\torsionN(\vts\cdot\vts,\Ahat)\vert^2+\gamma_2\mathring\chi\frac{\vert\mathring\sffN\vert^2}{H^2}\vert\cd H\vert^2.
\eann
Appealing to \eqref{eq:KatoA}, we get 
\bann
-\vert\cd\Ahat\vert^2-2\mnN \fff\Big(\nor\otimes\cd\frac{\sffN}{\mnN},\cd\Ahat\Big)\le{}&\left((1+\gamma_1^{-1}+\gamma_2^{-1}\mathring\chi)\frac{\vert\Ahat\vert^2}{H^2}+\frac{\vert\mathring\sffN\vert^2}{H^2}-\frac{n-1}{n(n+2)}\right)H^2\vert\torsionhat\vert^2\\
{}&+\bigg(\gamma_1+\gamma_2\mathring\chi\frac{n(n+2)}{2(n-1)}\frac{\vert\mathring\sffN\vert^2}{H^2}\bigg)\vert\cd\mathring\sffN+\torsionN(\vts\cdot\vts,\Ahat)\vert^2.
\eann
Now set
\[
\gamma_1=\beta_1\frac{\frac{1}{2}\vert\Ahat\vert^2}{W}\;\;\text{and}\;\;\gamma_2=\beta_2 \frac{2(n-1)}{n(n+2)}\frac{\frac{1}{2}\vert\Ahat\vert^2}{W}\frac{H^2}{\vert\mathring\sffN\vert^2}
\]
at points such that $\vert\mathring\sffN\vert^2 >0$, where $\beta_1$ and $\beta_2$ are to be chosen. If $\vert\mathring\sffN\vert^2=0$ then let $\gamma_1$ and $\beta_1$ be as before, and set $\gamma_2=\beta_2 = 1$. This yields 
\[\bigg(\gamma_1+\gamma_2\mathring\chi\frac{n(n+2)}{2(n-1)}\frac{\vert\mathring\sffN\vert^2}{H^2}\bigg)\vert\cd\mathring\sffN+\torsionN(\vts\cdot\vts,\Ahat)\vert^2 \leq (\beta_1+\beta_2\mathring\chi) \frac{\frac{1}{2}\vert\Ahat\vert^2}{W}\vert\cd\mathring\sffN+\torsionN(\vts\cdot\vts,\Ahat)\vert^2,\]
and consequently
\begin{align}
\label{eq:codim_bad_terms}
-\vert\cd\Ahat\vert^2-2\mnN \fff\Big({}&\nor\otimes\cd\frac{\sffN}{\mnN},\cd\Ahat\Big)\notag\\
\leq{}&\left(\frac{\vert\Ahat\vert^2}{H^2}+\frac{\vert\mathring\sffN\vert^2}{H^2}+\beta_1^{-1}\frac{2W}{H^2}+\frac{n(n+2)}{2(n-1)}\beta_2^{-1}\mathring\chi\frac{2W\vert\mathring\sffN\vert^2}{H^4}-\frac{n-1}{n(n+2)}\right)H^2\vert\torsionhat\vert^2\notag\\
{}&+\frac{\frac{1}{2}\vert\Ahat\vert^2}{W} \Bigg((\beta_1+\beta_2\mathring\chi) \vert\cd\mathring\sffN+\torsionN(\vts\cdot\vts,\Ahat)\vert^2\Bigg).
\end{align}
The terms on the second line can be controlled using the first good term in \eqref{eq:good grad terms} provided $\beta_1 + \beta_2 \mathring\chi < \alpha_1$; so let $\beta_1$ and $\beta_2$ satisfy this inequality. To handle the terms on the first line we use
\[
2W=bK+\tfrac{an-1}{n}H^2-\vert\Ahat\vert^2-\vert\mathring\sffN\vert^2\le bK+\tfrac{an-1}{n}H^2\vts
\]
to estimate
{\small
\bann
\frac{\vert\Ahat\vert^2}{H^2}{}&+\frac{\vert\mathring\sffN\vert^2}{H^2}+\beta_1^{-1}\frac{2W}{H^2}+\frac{n(n+2)}{2(n-1)}\beta_2^{-1}\mathring\chi\frac{2W\vert\mathring\sffN\vert^2}{H^4}\\
={}&(1-\beta_1^{-1})\frac{\vert\Ahat\vert^2}{H^2}+(1-\beta_1^{-1})\frac{\vert\mathring\sffN\vert^2}{H^2}+\beta_1^{-1}\left(\frac{bK}{H^2}+\frac{an-1}{n}\right)+\frac{n(n+2)}{2(n-1)}\beta_2^{-1}\mathring\chi\frac{2W\vert\mathring\sffN\vert^2}{H^4}\\
\le{}&(1-\beta_1^{-1})\frac{\vert\Ahat\vert^2}{H^2}+(1-\beta_1^{-1})\frac{\vert\mathring\sffN\vert^2}{H^2}+\beta_1^{-1}\!\!\left(\frac{bK}{H^2}+\frac{an-1}{n}\right)+\frac{n(n+2)}{2(n-1)}\beta_2^{-1}\mathring\chi\!\left(\frac{bK}{H^2}+\frac{an-1}{n}\right)\!\!\frac{\vert\mathring\sffN\vert^2}{H^2}\\
={}&(1-\beta_1^{-1})\frac{\vert\Ahat\vert^2}{H^2}+\left(\!1-\beta_1^{-1} +\beta_2^{-1}\mathring\chi\frac{(n+2)}{2(n-1)}\!\!\left[\frac{nbK}{H^2}+an-1\right]\!\right)\!\!\frac{\vert\mathring\sffN\vert^2}{H^2}+\beta_1^{-1}\!\!\left(\frac{bK}{H^2}+\frac{an-1}{n}\right)\!.
\eann}
Since $\beta_1<\alpha_1<1$, the first of these terms is non-positive. We will choose $\beta_1$ and $\beta_2$ so that the inequalities 
\begin{equation}
\label{eq:codim_grad_1}
1-\beta_1^{-1} +\beta_2^{-1}\mathring\chi \frac{(n+2)}{2(n-1)}\left[\frac{nbK}{H^2}+an-1\right] < 0
\end{equation}
and
\begin{equation}
\label{eq:codim_grad_2}
\beta_1^{-1}\left(\frac{bK}{H^2}+\frac{an-1}{n}\right)-\frac{n-1}{n(n+2)}<0
\end{equation}
hold whenever $H^2$ is sufficiently large compared to $K$. Together, these inequalities imply the first term on the right-hand side of \eqref{eq:codim_bad_terms} is non-positive.

To simplify notation, define $w\doteqdot\tfrac{(n+2)(an-1)}{2(n-1)}$. Let us set $\beta_1 = 2  w + \delta$, where $\delta >0$ is a small constant depending on $n$ and $m-\alpha$ which we progressively refine. Note that, for all $n \geq 5$, we can indeed choose $\delta$ so that $\beta_1 < \alpha_1$. We compute 
\bann
\beta_1^{-1}\left(\frac{bK}{H^2}+\frac{an-1}{n}\right)-\frac{n-1}{n(n+2)}={}&\frac{\beta_1^{-1}(n-1)}{n(n+2)}\left(\frac{n(n+2)}{n-1} \frac{bK}{H^2} -\delta \right),
\eann
hence there is a constant $C_0 = C_0(n)$ such that the right-hand side is negative (meaning \eqref{eq:codim_grad_2} holds) whenever 
\[\frac{H^2}{bK} \geq C_0\delta^{-1}.\]

Now we turn to \eqref{eq:codim_grad_1}. If $\mathring\chi=0$ we are done, so assume we are at a point where $\mathring\chi=1$. Observe that $w = 1-\alpha_1$. Hence the inequality $\beta_1 + \beta_2 < \alpha_1$ holds if 
\[\beta_2 < \alpha_1-\beta_1 = 1-3w -\delta.\]
For this reason we set $\beta_2 = 1-3w-2\delta$ (note this quantity is strictly positive for all $n \geq 5$ and $\delta$ suffciently small relative to $n$). We compute 
\bann
1-\beta_1^{-1} +\beta_2^{-1}\frac{(n+2)(an-1)}{2(n-1)} ={}&\beta_1^{-1} \beta_2^{-1}(\beta_1 \beta_2 -\beta_2 +\beta_1w)\\
={}& \beta_1^{-1} \beta_2^{-1}((4w-1)(1-w)  + \delta(3-6w-2\delta))\vts.
\eann
Now we make use of the pinching assumption to estimate the right-hand side. In dimensions $n \geq 8$, the inequality $a \leq \tfrac{4}{3n}$ ensures $w < \tfrac{1}{4}$; in dimensions $5 \leq n \leq 7$, we need $a < \tfrac{3(n+1)}{2n(n+2)}$ to ensure $w<\tfrac{1}{4}$. Each of these conditions is ensured by \eqref{eq:codim_a}, so we may indeed write $w = \tfrac{1}{4} - \eta$, where $\eta$ is a small positive number depending on $n$ and $m-\alpha$. Choosing $\delta$ sufficiently small relative to $\eta$ and $n$ then gives 
\bann
1-\beta_1^{-1} +\beta_2^{-1}\frac{(n+2)(an-1)}{2(n-1)} \leq {}& -\beta_1^{-1} \beta_2^{-1}\alpha_1\eta,
\eann
and consequently,
\bann
1-\beta_1^{-1} +\beta_2^{-1}\frac{(n+2)}{2(n-1)}\left[\frac{nbK}{H^2}+an-1\right] \leq \beta_2^{-1} \left[\frac{n(n+2)}{2(n-1)} \frac{bK}{H^2} -\beta_1^{-1}\alpha_1\eta \right]\vts.
\eann
In particular, there is a constant $C_1 = C_1(n)$ such that, as long as $\delta$ is small enough relative to $n$, \eqref{eq:codim_grad_1} holds at every point where 
\[\frac{H^2}{bK} \geq C_1\eta^{-1}.\]

To recap, there are constants $\delta \in (0,1)$ and $\eta \in (0, \tfrac{1}{4})$ depending on $n$ and $m-\alpha$ such that, at points where 
\begin{equation}
\label{eq:codim_H_large}
\frac{H^2}{bK} \geq \Lambda \doteqdot \max \{C_0\delta^{-1} , C_1\eta^{-1}\}\vts,
\end{equation}
the inequalities \eqref{eq:codim_grad_1} and \eqref{eq:codim_grad_2} hold. Combining these inequalities with \eqref{eq:codim_bad_terms}, we see that 
\bann
-\vert\cd\Ahat\vert^2-2\mnN \fff\Big(\nor\otimes{}\cd\frac{\sffN}{\mnN},\cd\Ahat\Big)\leq{}&\frac{\frac{1}{2}\vert\Ahat\vert^2}{W} \Bigg((\beta_1+\beta_2) \vert\cd\mathring\sffN+\torsionN(\vts\cdot\vts,\Ahat)\vert^2\Bigg)
\eann 
whenever \eqref{eq:codim_H_large} holds. Inserting $\beta_1 = 2w + \delta $ and $\beta_2 = 1-3w -2\delta$, and combining this estimate with \eqref{eq:good grad terms}, we arrive at 
\begin{equation*}
-\vert\cd\Ahat\vert^2-2\mnN \fff\Big(\nor\otimes\cd\frac{\sffN}{\mnN},\cd\Ahat\Big)\le \frac{\alpha_1 - \delta}{\alpha_1} \frac{\frac{1}{2}\vert\Ahat\vert^2}{W}\big(\vert\cd\sff\vert^2-a\vert\cd^\perp\mn\vert^2\big).
\end{equation*}
Hence the claim holds with $\theta \doteqdot 1-\alpha_1^{-1}\delta$.
\end{proof}

Since, at points where $\mnN>0$,
\bann
\frac{(\pd_t-\Delta)f_\sigma}{f_\sigma}={}&\frac{(\pd_t-\Delta)\vert\Ahat\vert^2}{\vert\Ahat\vert^2}-(1-\sigma)\frac{(\pd_t-\Delta)W}{W}+2(1-\sigma)\inner{\frac{\cd f_\sigma}{f_\sigma}}{\frac{\cd W}{W}}\\
&-\sigma(1-\sigma)\frac{\vert\cd W\vert^2}{W^2},
\eann
Claims \ref{claim:Naff reaction} and \ref{claim:codimension estimate gradient terms} imply that, for $\sigma=\sigma(n,m-\alpha)$ sufficiently small,
\ba
\frac{(\pd_t-\Delta)f_\sigma}{f_\sigma}\leq{}&2(1-\sigma)\inner{\frac{\cd f_\sigma}{f_\sigma}}{\frac{\cd W}{W}}\label{eq:Naff inequality}
\ea
wherever $\mnN^2\geq\Lambda K$. Now, since $2W\ge \varepsilon b(H^2+K)$, where $\varepsilon=\varepsilon(n,m-\alpha)$, we can estimate
\[
\frac{\frac{1}{2}\vert\Ahat\vert^2}{W}\le \frac{\frac{1}{2}(\vert\Ahat\vert^2+\vert\mathring\sffN\vert^2+\frac{1}{n}H^2)}{W}=\frac{a\mnN^2+bK-2W}{2W}\le C\vts,
\]
where $C=C(n,m-\alpha)$. This implies that $H^2$ is large compared to $K$ at any point where $f_\sigma$ is large compared to $K^\sigma$. Indeed, if
\[
f_\sigma\ge C\left(a\Lambda+b\right)K^\sigma\vts,
\]
then, since $\sigma<1$,
\[
H^2\ge \Lambda K\vts.
\]
We may now conclude from \eqref{eq:Naff inequality} that
\[
f_\sigma\le \max\left\{\max_{M\times\{0\}}f_\sigma,C\left(a\Lambda+b\right)K^\sigma\right\}\vts.
\]
The proposition follows.
\end{proof}

\subsection{The cylindrical estimate}\label{ssec:cylindrical estimate}

Next, we use the codimension estimate (Proposition \ref{prop:codimension estimate}) and the Poincar\'e-type inequality (Proposition \ref{prop:Poincare}) to obtain a cylindrical estimate. The argument follows the Stampacchia iteration procedure developed by Huisken \cite{Hu84}.

Given $n\ge 5$ and $\ell\ge 2$, it will be convenient to  state the estimate in terms of the class $\mathcal{C}_K^{n,\ell}(\alpha,V,\Theta)$ of $n$-dimensional submanifolds of $S_K^{n+\ell}$ satisfying
\begin{itemize}
\item $\vert\sff\vert^2-\tfrac{1}{n-2+\alpha}\vert\mn\vert^2-2(2-\alpha)K$, (with $\alpha\ge \alpha_n$, where $\alpha_n$ is defined by \eqref{eq:alpha_n}),
\item $\mu(M)\le VK^{-\frac{n}{2}}$, and
\item $\max_{M}\vert\sff\vert^2\le \Theta K$.
\end{itemize} 
Of course, every $n$-dimensional submanifold of $S_K^{n+\ell}$ satisfying \eqref{eq:strict quadratic pinching} lies in the class $\mathcal{C}_K^{n,\ell}(\alpha,V,\Theta)$ for some $\alpha\ge \alpha_n$, $V<\infty$, and $\Theta<\infty$.

\begin{proposition}[Cylindrical estimate]\label{prop:cylindrical estimate}
Let $X:M\times[0,T)\to S_K^{n+\ell}$ be a solution to mean curvature flow with initial condition in the class $\mathcal{C}_K^{n,\ell}(\alpha,V,\Theta)$ and set $\eta_0\doteqdot \frac{1}{n-2+\alpha_n}-\frac{1}{n-1}$. For every $\eta\in(0,\eta_0)$ there exists $C_\eta=C_\eta(n,\alpha,V,\Theta,\eta)<\infty$ so that
\begin{align}\label{eq:cylindrical estimate}
|\sff|^2-\tfrac{1}{n-1}|\mn|^2 \leq \eta \vert\mn\vert^2+C_\eta K\vts\mathrm{e}^{-2Kt} \quad\text{in}\quad M^n\times[0,T)\,.
\end{align}
\end{proposition}
\begin{proof}
Given $\eta>0$ and $\sigma\in(0,1)$, consider the function
\[
g_{\eta,\sigma}\doteqdot \frac{1}{2}\big[\vert \sff\vert^2-\big(\tfrac{1}{n-1}+\eta\big)\vert\mn\vert^2\big]W^{\sigma-1},
\]
where
\[
W\doteqdot \frac{1}{2}\big(bK+a\vert\mn\vert^2-\vert\sff\vert^2\big)
\]
for
\[
a=\tfrac{1}{n-2+\alpha_n}\;\;\text{and}\;\; b=2(2-\alpha_n).
\]
Computing as in the proof of Proposition \ref{prop:cylindrical decay}, we obtain, for $\eta\leq \eta_0\doteqdot \frac{1}{n-2+\alpha_n}-\frac{1}{n-1}$,
\ba\label{eq:evolve ges}
\frac{(\pd_t-\Delta)g_{\eta,\sigma}}{g_{\eta,\sigma}}\le{}&-2K+\sigma n\vert\sff\vert^2+2(1-\sigma)\inner{\frac{\cd g_{\eta,\sigma}}{g_{\eta,\sigma}}}{\frac{\cd W}{W}}\vts.
\ea

Now consider, for $k\ge 0$ and $p\ge 3$,
\[
v_k\doteqdot (\mathrm{e}^{2Kt}g_{\eta,\sigma}-k)_+^{\frac{p}{2}}\;\;\text{and}\;\; V_k(t)\doteqdot \spt v_k(\cdot,t)\vts.
\]
%
%
By \eqref{eq:evolve ges},
\bann
\frac{(\pd_t-\Delta)v_k^2}{v_k^2}={}&p\mathrm{e}^{2Kt}\frac{(\pd_t-\Delta)g_{\eta,\sigma}+2Kg_{\eta,\sigma}}{(\mathrm{e}^{2Kt}g_{\eta,\sigma}-k)_+}-4\big(1-\tfrac{1}{p}\big)\frac{\vert\cd v_k\vert^2}{v_k^2}\\
\le{}&\left(\sigma np\vert\sff\vert^2-2pK-2\gamma p\frac{\vert\cd\sff\vert^2}{W}\right)\frac{\mathrm{e}^{2Kt}g_{\eta,\sigma}}{(\mathrm{e}^{2Kt}g_{\eta,\sigma}-kK^\sigma)_+}+4\frac{\vert\cd v_k\vert}{v_k}\frac{\vert\cd W\vert}{W}\\
{}&-4\big(1-\tfrac{1}{p}\big)\frac{\vert\cd v_k\vert^2}{v_k^2}
\eann
and hence, for $p$ sufficiently large,
\bann
(\pd_t-\Delta)v_k^2\le{}&-2pKv_k^2+\sigma np\vts \mathrm{e}^{2Kt}g_{\eta,\sigma,+}(\mathrm{e}^{2Kt}g_{\eta,\sigma,+}-k)_+^{p-1}\vert\sff\vert^2-\gamma pv_k^2\frac{\vert\cd\sff\vert^2}{W}-2\vert\cd v_k\vert^2\vts.
\eann
It follows that
\begin{align}
\frac{d}{dt}\int v_k^2\,d\mu+2\int|\nabla v_k|^2\,d\mu+\int\vert\mn\vert^2v_k^2\,d\mu\leq{}& -2pK\int v_k^2\vts d\mu-\gamma p\int v_k^2\frac{\vert\cd\sff\vert^2}{W}\vts d\mu\nonumber\\
{}&+c\sigma p\int_{V_k}(\mathrm{e}^{2Kt}g_{\eta,\sigma})^pW\vts d\mu\vts,\label{eq:Stampacchia}
\end{align}
where $c=c(n)$. In particular, when $k=0$, the Poincar\'e inequality (Proposition \ref{prop:Poincare}) yields
\begin{align}\label{eq:v_k monotonicity}
\frac{d}{dt}\int v_0^2\,d\mu\leq{}& -2pK\int v_0^2\vts d\mu
\end{align}
and hence
\begin{align}\label{eq:L2est}
\int v_{0}^2\vts d\mu\leq C\mathrm{e}^{-2pKt}\vts,
\end{align}
so long as $p\ge\ell^{-1}$ and $\sigma\le \ell p^{-\frac{1}{2}}$, where $\ell=\ell(n,\alpha,\eta)$ and $C=C(n,K,\alpha,V,\Theta,\sigma,p)$.

The $L^2$-estimate \eqref{eq:L2est} can be bootstrapped to an $L^\infty$-estimate using \eqref{eq:Stampacchia} by applying Huisken--Stampacchia iteration. 
Indeed, using \eqref{eq:L2est} we may estimate, for $p\ge \ell$, $\sigma\le \ell^{-1}p^{-\frac{1}{2}}$,
\bann
V_k\leq{}&k^{-p}\int_{V_k}(\mathrm{e}^{2Kt}g_{\eta,\sigma})^{p}d\mu\le Ck^{-p}\vts,
\eann
where $C=C(n,K,\alpha,V,\Theta,\sigma,p)$, so that 
\bann
V_k\leq{}& \frac{\omega_n}{n+1}K^{-\frac{n}{2}}
\eann
so long as $k\ge k_0=k_0(n,K,\alpha,V,\Theta,\sigma,p)$. This enables us to apply the Sobolev inequality \cite{HoSp74} (cf. \cite{MiSi73}), which yields, at each time,
\[\frac{1}{c_{\mathrm{S}}} \Bigg(\int_{V_k} v_k^{2^*}d\mu\Bigg)^\frac{1}{2^*} \leq \int_{V_k}\big(\vert\cd v_k\vert^2 + v_k^2|\mn|^2\big)d\mu\vts,\]
where $2^*\doteqdot\tfrac{2n}{n-2}$ and $c_{\mathrm{S}}$ depends only on $n$, and consequently 
\bann
\frac{d}{dt}\int v_k^2\,d\mu+\frac{1}{c_{\mathrm{S}}} \Bigg(\int_{V_k} v_k^{2^*}d\mu\Bigg)^\frac{1}{2^*}
\le{}&c\vts\sigma p\int_{V_k}(\mathrm{e}^{2Kt}g_{\eta,\sigma})^{p}\big(\vert\mn\vert^2+K\big)\vts d\mu
\eann
so long as  $p\ge \ell$, $\sigma\le \ell^{-1}p^{-\frac{1}{2}}$, and $k\ge k_0$. Assuming further that $k_0\geq \max g_{\eta,\sigma}(\cdot,0)$ (which can be achieved with dependence on $n$, $K$, $\alpha$, $V$, $\Theta$, $\sigma$ and $p$ only), integrating now gives 
\begin{align}\label{eq:still holds with surgeries}
\sup_{t \in [0,T)} \int_M v_k^2(\cdot, t) \, d\mu +\int_0^T\! \bigg( \int_M v_k^{2q} \, d\mu \bigg)^\frac{1}{q} \,dt \leq C\sigma p \int_0^T\!\!\!\int_{V_k} (\mathrm{e}^{2Kt}g_{\eta,\sigma})^{p}\big(\vert\mn\vert^2+K\big)\vts \,d\mu \,dt,
\end{align}
where $C\doteqdot c_{\mathrm{S}}c$ depends only on $n$. Applying the H\"older, interpolation and Young inequalities, exactly as in the proof of \cite[Theorem 5.1]{Hu84}, now yields $C=C(n,\ell,\alpha,V,\Theta,\sigma,p)<\infty$ and $\gamma=\gamma(n)>0$ such that
\begin{align*}
A(h) \leq \frac{C}{(h-k)^p} A(k)^{\gamma}
\end{align*}
for all $h>k>k_0$, where
\[
A(k)\doteqdot \int_0^T\!\!\!\int_{V_k} \,d\mu_t\vts dt\vts,
\]
so long as $p\ge \ell^{-1}$ and $\sigma\le\ell p^{-\frac{1}{2}}$, where $\ell=\ell(n,\alpha,\eta)$. Fixing $p\doteqdot \ell^{-1}$ and $\sigma\doteqdot \ell p^{-\frac{1}{2}}$, Stampacchia's lemma \cite[Lemma 4.1]{St66} now yields
\begin{align*}
A(k_0+d)=0\,,
\end{align*}
where
\begin{align*}
d^p\doteqdot 2^{p\gamma/(\gamma-1)}CA(k_0)^{\gamma-1}\,.
\end{align*}
Estimating via \eqref{eq:L2est}
\begin{align*}
A(k)\leq k_0^{-p}\int_0^T\hspace{-2mm}\int v_{k_0}^2\,d\mu\,dt\leq C(n,K,\alpha,V,\Theta,\eta)\vts,
\end{align*}
we conclude that
\begin{align*}
\mathrm{e}^{2Kt}g_{\eta,\sigma}\leq C(n,K,\alpha,V,\Theta,\eta)\,.
\end{align*}
Young's inequality then yields
\begin{align*}
|\sff|^2-\tfrac{1}{n-1}\vert\mn\vert^2\leq 2\eta \vert\mn\vert^2+C(n,K,\alpha,V,\Theta,\eta)\vts\mathrm{e}^{-2K t}\,.
\end{align*}
The theorem now follows from the scaling covariance of the estimate. 
\end{proof}

\subsection{The gradient estimate} 

Next, we derive a suitable analogue of the ``gradient estimate" \cite[Theorem 6.1]{HuSi09}. We need the following doubling estimates for solutions with initial data in the class $\mathcal{C}^{n,\ell}_K(\alpha,V,\Theta)$.

\begin{proposition}\label{prop:class C universal interior estimates}
Let $X:M\times [0,T)\to S_K^{n+\ell}$ be a maximal solution to mean curvature flow with initial condition in the class $\mathcal{C}_K^{n,\ell}(\alpha,V,\Theta)$. Defining $\Lambda_0$ and $\lambda_0$ by
\begin{equation}\label{eq:lambdas}
\Lambda_0/2\doteqdot \Theta\;\;\text{and}\;\;\mathrm{e}^{2n\lambda_0}\doteqdot 1+\frac{n}{n+3\Lambda_0},
\end{equation}
we have
\begin{equation}\label{eq:class C universal T bound}
\mathrm{e}^{2nKT}\ge 1+\frac{2n}{3\Lambda_0}\,,
\end{equation}
and, for every $k\in\mathbb{N}$,
\begin{equation}\label{eq:class C universal interior estimates}
\max_{M\times\{\lambda_0K^{-1}\}}\vert\cd^k{\sff}\vert^2\leq \Lambda_kK^{k+1}\vts,
\end{equation}
where $\Lambda_k$ depends only on $n$, $k$ and $\Theta$.
\end{proposition}
\begin{proof}
Since
\[
\max_{M\times\{0\}}\vert{\sff}\vert^{2}\leq \Lambda_0K/2\,,
\]
a straightforward \textsc{ode} comparison argument applied to the inequality
\[
(\pd_t-\Delta)\vert{\sff}\vert^2\leq (3\vert\sff\vert^2+2nK)\vert\sff\vert^2
\]
yields
\[
\max_{M\times\{t\}}\vert{\sff}\vert^{2}\leq\frac{nK}{\left(3+\frac{2n}{\Lambda_0}\right)\mathrm{e}^{-2nKt}-3}\,.
\]
We immediately obtain \eqref{eq:class C universal T bound} and
\ba\label{eq:curvature_bounded_gradient_estimate}
\vert{\sff}\vert^2(\vts\cdot\vts,t)\leq \Lambda_0K\;\;\text{for all}\;\; t\leq \lambda_0K^{-1}\,.
\ea
The claim \eqref{eq:class C universal interior estimates} now follows from the Bernstein estimates (Proposition \ref{prop:Bernstein}).
\end{proof}

Modifying an argument of Huisken \cite[Theorem 6.1]{Hu84} and Huisken--Sinestrari \cite[Theorem 6.1]{HuSi09}, we can now obtain a pointwise estimate for the gradient of the second fundamental form which holds up to the singular time.

\begin{proposition}[Gradient estimate (cf. {\cite[Theorem 6.1]{HuSi09}})]\label{prop:gradient estimate}
Let $X:M\times[0,T)\to S_K^{n+\ell}$, $n\geq 5$, be a solution to mean curvature flow with initial condition in the class $\mathcal{C}_K^{n,\ell}(\alpha,V,\Theta)$. There exist $c=c(n,\ell,\alpha,\Theta)<\infty$, $\eta_0=\eta_0(n)>0$ and, for every $\eta\in(0,\eta_0)$, $C_\eta=C_\eta(n,\alpha,V,\Theta,\eta)<\infty$ such that
\begin{equation}\label{eq:gradient estimate}
\frac{\vert\cd \sff\vert^2}{\vert\mn\vert^2+K}+c\left(\vert{\sff}\vert^2-\frac{1}{n-1}\vert\mn\vert^2\right)\le\left(\eta+\frac{1}{n-1}\right)\vert\mn\vert^2+ C_\eta K\mathrm{e}^{-2Kt}
\end{equation}
in $M\times [\lambda_0K^{-1},T)$, where $\lambda_0$ is defined by \eqref{eq:lambdas}.
\end{proposition}
Note that the conclusion is not vacuous since, by Proposition \ref{prop:class C universal interior estimates}, the maximal existence time of a solution with initial data in the class $\mathcal{C}_K^n(\alpha,V,\Theta)$ is at least $\frac{1}{2nK}\log\left(1+\frac{2n}{3\Lambda_0}\right)>\lambda_0 K^{-1}$. 

Setting $\eta=\eta_0$, say, yields the cruder estimate
\begin{equation}\label{eq:gradient estimate eta=1}
\vert\cd \sff\vert^2\leq C(\vert\mn\vert^4+K^2)\vts,
\end{equation}
where $C=C(n,\alpha,V,\Theta)$.

\begin{proof}[Proof of Proposition \ref{prop:gradient estimate}]
We proceed as in \cite[Theorem 6.1]{HuSi09}. By \eqref{eqn_evolderiv},
\bann
(\partial_t- \Delta)|\nabla \sff|^2 \leq{}& -2 |\nabla^2 \sff|^2 + c\lb |\sff|^2+K\rb|\nabla \sff|^2\,.
\eann
We will control the bad term using the good term in the evolution equation for $\vert\sff\vert^2$ and the Kato inequality \eqref{eq:Kato inequality}.
 
By the cylindrical estimate, given any $\eta>0$ we can find $C_\eta=C_\eta(n,\alpha,V,\Theta,\eta)>2$ such that
\[
\vert\sff\vert^2-\tfrac{1}{n-1}\vert\mn\vert^2\leq \eta \vert\mn\vert^2+C_\eta K\mathrm{e}^{-2Kt}\,,
\]
and hence
\[
G_\eta\doteqdot 2C_\eta K\mathrm{e}^{-2Kt}+\left(\eta+\tfrac{1}{n-1}\right)\vert\mn\vert^2-\vert\sff\vert^2\geq C_\eta K\mathrm{e}^{-2Kt}>0\,.
\]
By the initial pinching condition,
\[
G_0\doteqdot 4(2-\alpha_n)K+\tfrac{1}{n-2+\alpha_n}\vert\mn\vert^2-\vert\sff\vert^2\geq 2(2-\alpha_n)K>0\,.
\]

Arguing as in \eqref{eq:evolve g eta H nonzero} and \eqref{eq:evolve g eta H zero}, but keeping part of the gradient terms using the Kato inequality, we obtain
\bann
(\pd_t-\Delta)G_\eta\ge{}&-2n\vert\sff\vert^2(2C_\eta K\mathrm{e}^{-2Kt}-G_\eta)-4C_\eta K^2\mathrm{e}^{-2Kt}+\beta\vert\cd\sff\vert^2\\
\ge{}&-2n(\vert{\sff}\vert^2+K)G_\eta+\beta\vert\cd\sff\vert^2\vts,
\eann
where $\beta\doteqdot \frac{ 3}{n+2}-\frac{1}{n-1}=\frac{2n-5}{(n+2)(n-1)}$, so long as $\eta\leq \left(1-\frac{3}{n+2}\right)\beta$.

Similarly,
\bann
(\pd_t-\Delta)G_0
\geq{}&-2n(\vert{\sff}\vert^2+K)G_0\,.
\eann

We seek a bound for the ratio $\frac{\vert\cd \sff\vert^2}{G_\eta G_0}$. So suppose that $\frac{\vert \cd \sff \vert^2}{G_\eta G_0}$ attains a (parabolic) interior local maximum at $(x_0,t_0)$. Then, at $(x_0,t_0)$,
\bann
0=\cd_k\frac{\vert\cd  \sff \vert^2}{G_\eta G_0}=2\frac{\inner{\cd_k\cd  \sff }{\cd  \sff }}{G_\eta G_0}-\frac{\vert\cd  \sff \vert^2}{G_\eta G_0}\lb\frac{\cd_k G_\eta}{G_\eta}+\frac{\cd_kG_0}{G_0}\rb
\eann
and hence
\bann
4\frac{\vert\cd  \sff \vert^2}{G_\eta G_0}\inner{\frac{\cd G_\eta}{G_\eta}}{\frac{\cd G_0}{G_0}}\leq{}&\frac{\vert\cd  \sff \vert^2}{G_\eta G_0}\left\vert\frac{\cd G_\eta}{G_\eta}+\frac{\cd G_0}{G_0}\right\vert^2\leq4\frac{\vert\cd^2  \sff \vert^2}{G_\eta G_0}\,.
\eann
Furthermore, at $(x_0,t_0)$,
\bann
0\leq (\pd_t-\Delta)\frac{\vert\cd  \sff \vert^2}{G_\eta G_0}={}&\frac{(\pd_t-\Delta)\vert\cd  \sff \vert^2}{G_\eta G_0}-\frac{\vert\cd  \sff \vert^2}{G_\eta G_0}\lb\frac{(\pd_t-\Delta)G_\eta}{G_\eta}+\frac{(\pd_t-\Delta)G_0}{G_0}\rb\\
{}&+\frac{2}{G_\eta G_0}\inner{\cd\frac{\vert\cd  \sff \vert^2}{G_\eta G_0}}{\cd(G_\eta G_0)}+2\frac{\vert\cd  \sff \vert^2}{G_\eta G_0}\inner{\frac{\cd G_\eta}{G_\eta}}{\frac{\cd G_0}{G_0}}\\
\leq{}&\frac{\vert\cd  \sff \vert^2}{G_\eta G_0}\lb (c+4n)(\vert \sff \vert^2+K)-\beta\frac{\vert\cd  \sff \vert^2}{G_\eta}\rb
\eann
and hence
\bann
\frac{\vert\cd  \sff \vert^2}{G_\eta G_0}\leq{}& \frac{(c+4n)\beta^{-1}(\vert \sff \vert^2+K)}{4(2-\alpha_n)K+\frac{1}{n-2+\alpha_n}\vert\mn\vert^2-\vert \sff \vert^2}.
\eann
Since the solution is uniformly pinched,
\[
\vert \sff \vert^2\le \tfrac{1}{n-2+\alpha}\vert\mn\vert^2+2(2-\alpha)K\,,
\]
we obtain, at $(x_0,t_0)$,
\bann
\frac{\vert\cd  \sff \vert^2}{G_\eta G_0}
\le{}&C\,,
\eann
where $C$ depends only on $n$, $k$ and $\alpha$.

On the other hand, since $G_0>C_0K\doteqdot 2(2-\alpha_n)K$ and $G_\eta>C_\eta K\mathrm{e}^{-2Kt}$, if no interior local parabolic maxima are attained, then, by Proposition \ref{prop:class C universal interior estimates}, we have for any $t\ge \lambda_0K^{-1}$
\bann
\max_{\M\times \{t\}}\frac{\vert\cd{\sff }\vert^2}{G_0G_\eta}\leq{}&\max_{\M\times \{\lambda_0K^{-1}\}}\frac{\vert\cd{\sff }\vert^2}{G_0G_\eta}\\
\leq{}&\max_{\M\times \{\lambda_0K^{-1}\}}\frac{\vert\cd{\sff }\vert^2}{C_0C_\eta K^2\mathrm{e}^{-2\lambda_0}}\\
\leq{}&\frac{\Lambda_1\mathrm{e}^{2\lambda_0}}{C_0C_\eta}\\
\leq{}&\Lambda_1\mathrm{e}^{2\lambda_0}\,.
\eann
The theorem follows.
\end{proof}

\subsection{Higher order estimates}

The gradient estimate can be used to bound the first order terms which arise in the evolution equation for $\cd^2{\sff}$. A straightforward maximum principle argument exploiting this observation yields an analogous estimate for $\cd^2{\sff}$.

\begin{proposition}[Hessian estimate (cf. \cite{Hu84,HuSi09})]\label{prop:Hessian estimate}
Let $X:M\times[0,T)\to S_K^{n+1}$, $n\geq 5$, be a solution to mean curvature flow with initial condition in the class $\mathcal{C}_K^{n,\ell}(\alpha,V,\Theta)$. There exists $C=C(n,\alpha,V,\Theta)$ such that
\ba\label{eq:scale_invariant_Hessian_estimate}
\vert\cd^2{\sff}\vert^2\leq C(\vert\mn\vert^{6}+K^{3})\;\;\text{in}\;\; M \times[\lambda_0K^{-1},T)\,.
\ea
\end{proposition}
\begin{proof}
We proceed as in \cite[Theorem 6.3]{HuSi09}. Set $W\doteqdot \vert\mn\vert^2+K$. By \eqref{eqn_evolderiv2},
\[
(\pd_t-\Delta)\vert\cd^2{\sff}\vert^2\leq c\left(W\vert\cd^2{\sff}\vert^2+\vert\cd{\sff}\vert^4\right)-2\vert\cd^3{\sff}\vert^2\,,
\]
where $c$ depends only on $n$ and $k$. Recalling \eqref{eq:evolve H squared}, we obtain
\bann
(\pd_t-\Delta)\frac{\vert\cd^2{\sff}\vert^2}{W^\frac{5}{2}}\leq{}& \frac{c}{W^\frac{5}{2}}\lsb W\vert\cd^2{\sff}\vert^2+\vert\cd{\sff}\vert^4\rsb-2\frac{\vert\cd^3{\sff}\vert^2}{W^\frac{5}{2}}\\
{}& \! -5\frac{\vert\cd^2{\sff}\vert^2}{W^\frac{7}{2}}\lsb \vert{\Wein(\cdot,\mn)}\vert^2+nK\vert\mn\vert^2-\vert\cd^\perp\mn\vert^2\rsb\\
{}&-\frac{35}{4}\frac{\vert\cd^2\sff\vert^2}{W^{\frac{7}{2}}}\frac{\vert \cd W\vert^2}{W}+\frac{5}{W^\frac{7}{2}}\inner{\cd\vert\cd^2{\sff}\vert^2}{\cd W}.
\eann

We can use the good third order term on the first line to absorb the ultimate term, since
\bann
\frac{5}{W^\frac{7}{2}}\inner{\cd\vert\cd^2{\sff}\vert^2}{\cd W}\le{}&\frac{10}{W^\frac{7}{2}}\vert\cd^3{\sff}\vert\vert\cd^2\sff\vert\vert\cd W\vert\\
\le {}&\frac{1}{W^\frac{1}{2}}\left(\frac{\vert\cd^3{\sff}\vert^2}{W^2}+25\frac{\vert\cd^2\sff\vert^2\vert\cd W\vert^2}{W^4}\right).
\eann
Estimating
\[
\frac{\vert\cd W\vert^2}{W}\le 4\vert\cd^\perp\mn\vert^2
\]
then yields 
\bann
(\pd_t-\Delta)\frac{\vert\cd^2{\sff}\vert^2}{W^\frac{5}{2}}\leq{}& \frac{c}{W^\frac{5}{2}}\lsb W\vert\cd^2{\sff}\vert^2+\vert\cd{\sff}\vert^4\rsb-\frac{\vert\cd^3{\sff}\vert^2}{W^\frac{5}{2}}
+100\frac{\vert\cd^2{\sff}\vert^2}{W^\frac{7}{2}}\vert\cd^\perp\mn\vert^2.
\eann

Restricting to $t \geq \lambda_0 K^{-1}$ and estimating the first order terms using Proposition \ref{prop:gradient estimate} (and Young's inequality) then yields
\bann
(\pd_t-\Delta)\frac{\vert\cd^2{\sff}\vert^2}{W^{\frac{5}{2}}}\leq{}&c_1\frac{\vert\cd^2{\sff}\vert^2}{W^{\frac{3}{2}}}+C_1K^2\frac{\vert\cd^2{\sff}\vert^2}{W^{\frac{7}{2}}}\mathrm{e}^{-2Kt}\\
{}&+\frac{c_1\vert\mn\vert^4W^2+C_1K^4\mathrm{e}^{-4Kt}}{W^{\frac{5}{2}}}-\frac{\vert\cd^3{\sff}\vert^2}{W^{\frac{5}{2}}}\,,
\eann
where $c_1$ depends only on $n$, $k$, $\alpha$ and $\Theta$, and $C_1$ depends also on $V$.

Similar arguments yield
\bann
(\pd_t-\Delta)\frac{\vert\cd{\sff}\vert^2}{W^{\frac{3}{2}}}
\leq{}&\frac{c_2\vert\mn\vert^2W^3+C_2K^4\mathrm{e}^{-4Kt}}{W^{\frac{5}{2}}}-\frac{\vert\cd^2{\sff}\vert^2}{W^\frac{3}{2}}
\eann
and
\bann
(\pd_t-\Delta)\frac{\vert\cd{\sff}\vert^2}{W^{\frac{7}{2}}}
\le{}&c\frac{\vert\cd \sff\vert^2}{W^{\frac{9}{2}}}\left(W^2+\vert\cd^\perp\mn\vert^2\right)-\frac{\vert \cd^2\sff\vert^2}{W^{\frac{7}{2}}}\\
\leq{}&\frac{c_3\vert\mn\vert^2W^3+C_3K^4\mathrm{e}^{-4Kt}}{W^{\frac{9}{2}}}-\frac{\vert\cd^2{\sff}\vert^2}{W^\frac{7}{2}}\,,
\eann
where $c_2$ and $c_3$ depend only on $n$, $\alpha$, and $\Theta$, and $C_2$ and $C_3$ depend also on $V$.

Setting
\[
f\doteqdot\frac{\vert\cd^2{\sff}\vert^2}{W^{\frac{5}{2}}}+c_1\frac{\vert\cd{\sff}\vert^2}{W^{\frac{3}{2}}} +C_1K^2\frac{\vert\cd{\sff}\vert^2}{W^{\frac{7}{2}}}
\]
and estimating $W\ge K$, we obtain
\bann
(\pd_t-\Delta)f\leq{}&\frac{c_1\vert\mn\vert^4W^2+C_1K^4\mathrm{e}^{-4Kt}}{W^{\frac{5}{2}}}+c_1\frac{c_2\vert\mn\vert^2W^3+C_2K^4\mathrm{e}^{-4Kt}}{W^{\frac{5}{2}}}\\
{}&+C_1K^2\frac{c_3\vert\mn\vert^2W^3+C_3K^4\mathrm{e}^{-4Kt}}{W^{\frac{9}{2}}}\\
\le{}&\frac{(c_1+c_1c_2+c_3C_1)\vert\mn\vert^2W^3+(C_1+c_1C_2+C_1C_3)K^4\mathrm{e}^{-4Kt}}{W^{\frac{5}{2}}}\\
\le{}&(c_1+c_1c_2+c_3C_1)\vert\mn\vert^2W^{\frac{1}{2}}+(C_1+c_1C_2+C_1C_3)K^\frac{3}{2}\mathrm{e}^{-4Kt}\\
\doteqdot{}&c_4\vert\mn\vert^2W^{\frac{1}{2}}+C_4K^\frac{3}{2}\mathrm{e}^{-4Kt}\,.
\eann
On the other hand, the function $G$ defined by
 \[
G^2\doteqdot 2(2-\alpha_n)K+\tfrac{1}{n-2+\alpha_n}\vert\mn\vert^2-\vert\sff\vert^2
\]
satisfies
\[
(\pd_t-\Delta)G\ge c_5\vert\mn\vert^2W^{\frac{1}{2}}\vts,
\]
where $c_5$ depends only on $n$ and $\alpha$. Thus,
\bann
(\pd_t-\Delta)\lb f-\frac{c_4}{c_5}G+\frac{C_4}{4}K^{\frac{1}{2}}\mathrm{e}^{-4Kt}\rb\leq{}&0\,.
\eann
The maximum principle and Proposition \ref{prop:class C universal interior estimates} then yield
\bann
\max_{M\times\{t\}}\left(f-\frac{c_4}{c_5}G\right)\leq{}&\max_{M\times\{\lambda_0K^{-1}\}}\left(f-\frac{c_4}{c_5}G\right)+\frac{C_4}{4}K^{\frac{1}{2}}\lb\mathrm{e}^{-4 \lambda_0}-\mathrm{e}^{-4 Kt}\rb\\
\leq{}&C_5K^{\frac{1}{2}}
\eann
for all $t\geq \lambda_0K^{-1}$, where $C_5$ depends only on $n$, $\alpha$, $V$, and $\Theta$. We conclude that
\[
\vert\cd^2{\sff}\vert^2\leq cW^3+CK^{\frac{1}{2}}W^{\frac{5}{2}}\;\;\text{in}\;\; M\times [\lambda_0K^{-1},T)\,,
\]
where $c$ and $C$ depend only on $n$, $\alpha$, $V$, and $\Theta$. The claim now follows from Young's inequality.
\end{proof}

Applying the Hessian estimate in conjunction with the the rough evolution equation
\[
(\cd_t-\Delta)\sff=\sff\ast\sff\ast\sff+K\ast\sff
\] 
for $\sff$ yields an analogous bound for $\cd_t \sff$, and hence, in particular, for the time derivative of $\mn$. Thus, in high curvature regions, we obtain the following a priori bounds for $\cd^\perp\mn$ and $\cd_t^\perp\mn$.

\begin{corollary}\label{cor:spacetime grad H bound}
Let $X:M^n\times[0,T)\to S_K^{n+\ell}$, $n\geq 2$, be a solution to mean curvature flow with initial condition in the class $\mathcal{C}_K^{n,\ell}(\alpha,V,\Theta)$. There exist $h_\sharp=h_\sharp(n,\alpha,V,\Theta)$ and $c_\sharp=c_\sharp(n,\alpha,V,\Theta)$ such that
\begin{equation}
\vert\mn\vert(x,t)\geq h_\sharp \sqrt K
\;\implies \;
\frac{\vert \cd^\perp\mn\vert}{\vert\mn\vert^2}(x,t)\leq c_\sharp\;\;\text{and}\;\;\frac{\vert \cd^\perp_t\mn\vert}{\vert\mn\vert^3}(x,t)\leq \frac{c_\sharp^2}{2}.
\end{equation}
\end{corollary}

\subsection{Neck detection}

The cylindrical and gradient estimates can be used to show that regions of very high curvature which are not pinched in the sense of Andrews and Baker must form high quality `neck' regions.

\begin{definition}
Let $X: M\to S_K^{n+\ell}\subset \R^{n+\ell+1}$ be an immersed submanifold of $S_K^{n+\ell}$. A point $x\in M$ \emph{lies at the center of an} $(\varepsilon,k,L)$-\emph{neck of size $r$} if the map $\exp_{r^{-1}X(x)}^{-1}\circ (r^{-1}X)$ is $\varepsilon$-cylindrical and $(\varepsilon,k)$-parallel at all points in the induced intrinsic ball of radius $L$ about $x$ in the sense of \cite[Definition 3.9]{HuSi09}.
\end{definition}


\begin{lemma}[Neck detection (cf. {\cite[Lemma 7.4]{HuSi09}})]\label{lem:curvature neck detection}
Let $X:M\times[0,T)\to S_K^{n+\ell}$ be a solution to mean curvature flow with initial condition in the class $\mathcal{C}^{n,\ell}_K(\alpha,V,\Theta)$. Given $\varepsilon\leq \frac{1}{100}$, there exist parameters $\eta_\sharp=\eta_\sharp(n,\ell,\alpha,V,\Theta,\varepsilon)>0$ and $h_\sharp=h_\sharp(n,\ell,\alpha,V,\Theta,\varepsilon)<\infty$ with the following property. If
\bann
\vert\mn\vert(x_0,t_0)\ge h_\sharp\sqrt K\;\;\text{ and }\;\; \big(\vert\sff\vert^2-\tfrac{1}{n-1}\vert\mn\vert^2\big)(x_0,t_0)\ge -\eta_\sharp\vert\mn\vert^2(x_0,t_0)\,,
\eann
then 
\[
\Lambda_{r_0,k,\varepsilon}(x_0,t_0)\leq \varepsilon r_0^{-(k+1)}
\]
for each $k=0,\dots,\lfloor\frac{2}{\varepsilon}\rfloor$, where $r_0\doteqdot  \frac{n-1}{\vert\mn\vert(x_0,t_0)}$,
\[
\Lambda_{r,0,\varepsilon}(x,t)\doteqdot  \max_{\mathcal{B}_{\varepsilon^{-1}r}(x,t)\times(t-10^4r^2,t]}\big\vert\vert\sff\vert^2-\tfrac{1}{n-1}\vert\mn\vert^2\big\vert\,,
\]
and, for each $k\geq 1$,
\[
\Lambda_{r,k,\varepsilon}(x,t)\doteqdot  \max_{\mathcal{B}_{\varepsilon^{-1}r}(x,t)\times(t-10^4r^2,t]}\vert \cd^k{\sff}\vert\,.
\]
\end{lemma}
\begin{proof}
The result can be obtained by \emph{reductio ad absurdum}, exploiting the cylindrical, gradient and Hessian estimates for the second fundamental form. We do not include the argument since it is similar to that of \cite[Lemma 7.4]{HuSi09} (cf. \cite[Theorem 7.13]{BrHu17}, \cite[Lemma 4.16]{LangfordNguyen} and \cite[Lemma 5.5]{Nguyen2020}).
\end{proof}


By the Gauss equation and the arguments of \cite[\S 3]{Nguyen2020} (cf. \cite[\S 3]{HuSi09}), necks of sufficiently high quality can be integrated (after pulling up to the tangent space) to obtain ``almost hypersurface'' necks (in the tangent space), which can be replaced by a pair of ``convex caps'' in a controlled way.

\subsection{Hypersurface detection}

The codimension estimate can be used to show that, after pulling up to the tangent space, regions of high curvature almost lie in some $(n+1)$-dimensional affine subspace.

\begin{definition}
An immersed submanifold $X: M\to S_K^{n+\ell}\subset \R^{n+\ell+1}$ of $S_K^{n+\ell}$ is $(\varepsilon,k)$-\emph{almost hypersurface} about $x_0\in X$ if, for some $r>0$, the map $\exp_{r^{-1}X(x_0)}^{-1}\circ (r^{-1}X)$ is $(\varepsilon,k)$-almost hypersurface in the sense of \cite[Definition 3.1]{Nguyen2020}. That is, it satisfies
\[
\vert\cd^m\Ahat\vert\le \varepsilon \;\;\text{for each}\;\; m=0,\dots,k\vts.
\]
\end{definition}

The following lemma, combined with the Gauss equation, shows that points of sufficiently large curvature have almost hypersurface neighbourhoods (of `size' $\sim \vert \mn\vert^{-1}$).

\begin{lemma}[Hypersurface detection (cf. {\cite[Lemma 5.8]{Nguyen2020}})]\label{lem:hypersurface detection}
Let $X:M\times[0,T)\to S_K^{n+\ell}$ be a solution to mean curvature flow with initial condition in the class $\mathcal{C}^{n,\ell}_K(\alpha,V,\Theta)$. Given $\varepsilon\leq \frac{1}{100}$, there exist $h_\sharp=h_\sharp(n,\ell,\alpha,V,\Theta,\varepsilon)<\infty$, $L_\sharp=L_\sharp(n,\ell,\alpha,V,\Theta)>0$ and $\theta_\sharp=\theta_\sharp(n,\ell,\alpha,V,\Theta)>0$ with the following property. If
\bann
\vert\mn\vert(x_0,t_0)\ge h_\sharp\sqrt K\vts,
\eann
then 
\[
\sup_{\mathcal{B}_{L_\sharp r_0}(x_0,t_0)\times (t_0-\theta_\sharp r_0^2,t_0]}\vert\cd^k\Ahat\vert\le \varepsilon r_0^{-(k+1)}
\]
for each $k=0,\dots,\lfloor\frac{2}{\varepsilon}\rfloor$, where $r_0\doteqdot  \frac{n-1}{\vert\mn\vert(x_0,t_0)}$.
\end{lemma}
\begin{proof}
The proof is again very similar to that of \cite[Lemma 7.4]{HuSi09}.
\end{proof}


Note that almost hypersurface regions which satisfy our pinching condition do indeed lie close to a genuine ``hypersurface''. This can be proved by an argument similar to \cite[Proposition 2.4]{Naff} (cf. \cite[Theorem 6.3]{Nguyen2020}). We do not require this here, however. Indeed, we only need bounds for $\Ahat$ and $\cd\Ahat$ (note that, under the quadratic pinching condition, the torsion is controlled by $\cd\Ahat$).

\section{The key estimates for surgically modified flows}\label{sec:key estimates surgery}

We need to show that suitable versions of the key estimates still hold in the presence of surgeries. In the following definition, surgery is performed on the middle third of a neck of size $r$ in the obvious way: 
\begin{enumerate}[(i)]
\item Scale by $r^{-1}$ and precompose with $\exp_{r^{-1}X(p)}^{-1}$ to obtain a neck in $T_{r^{-1}X(p)}S_{r^2 K}^{n+\ell}$.
\item Perform surgery on the middle third of this neck in $T_{r^{-1}X(p)}S_{r^2K}^{n+\ell}$ as described in \cite[Section 3]{Nguyen2020} (cf. \cite[Section 3]{HuSi09}).
\item Re-embed in $S_{K}^{n+\ell}$ by composing with $\exp_{r^{-1}X(p)}$ and scaling by $r$.
\end{enumerate}
\begin{definition}
A \emph{surgically modified (mean curvature) flow} in $S_K^{n+\ell}$ with \emph{neck parameters} $(\varepsilon,k,L)$, \emph{surgery parameters} $(\tau,B)$, and \emph{surgery scale} $r$ is a finite sequence $\{X_i:M_i^n\times[T_i,T_{i+1}]\to S_K^{n+\ell}\}_{i=1}^{N-1}$ of smooth mean curvature flows $X_i:M_i^n\times[T_i,T_{i+1}]\to S_K^{n+\ell}$ for which the $(i+1)$-st initial datum $X_{i+1}(\vts\cdot\vts,T_{i+1}):M_{i+1}\to S_K^{n+\ell}$ is obtained from the $i$-th final datum $X_{i}(\vts\cdot\vts,T_{i+1}):M_{i}\to S_K^{n+\ell}$ by performing finitely many $(\tau,B)$-standard surgeries, in the sense of \cite[Section 3]{HuSi09}, on the middle thirds of $(\varepsilon,k,L)$-necks with mean curvature satisfying $\frac{n-1}{10r} \le H\le \frac{10(n-1)}{r}$, and then discarding finitely many connected components that are diffeomorphic either to $S^n$ or to $S^1\times S^{n-1}$.
\end{definition}

\subsection{Preserving quadratic pinching}

For a suitable range of neck and surgery parameters, and surgery scales, the surgery procedure of \cite{Nguyen2020} (cf. \cite[Section 3]{HuSi09}) preserves the quadratic pinching condition \eqref{eq:strict quadratic pinching}. 

{\bf In the statement of the following proposition (Proposition \ref{prop:preserving pinching with surgery}) and henceforth, when we refer to a surgically modified flow, it is taken for granted that the neck and surgery parameters, and the surgery scale, are fixed within a suitable range (which we progressively refine).}

\begin{proposition}[Quadratic pinching for surgically modified flows]\label{prop:preserving pinching with surgery}
Every surgically modified flow $\{X_i:M_i^n\times[T_i,T_{i+1}]\to S_K^{n+\ell}\}_{i=1}^{N-1}$, $n\ge 5$, with initial condition in the class $\mathcal{C}_K^{n,\ell}(\alpha,V,\Theta)$ (with $\alpha>\alpha_n$ when $n=5,6,7$) satisfies \eqref{eq:uniform quadratic pinching} for all $t\in[T_1,T_N]$.
\end{proposition}
\begin{proof}
Since $|\sff|^2\equiv \frac{1}{n-1}|\mn|^2$ on a hypersurface cylinder and $|\sff|^2 \equiv \frac{1}{n}\vert\mn\vert^2$ on a hypersurface cap, we can ensure, for a suitable choice of neck and surgery parameters and surgery scales, that
\[
\vert\sff\vert^2-\frac{1}{n-2+\overline\alpha}\vert\mn\vert^2<2(2-\overline\alpha)K
\]
on regions modified or added by surgery, where $\overline\alpha\doteqdot\frac{\alpha+10}{11}\in(\alpha,1)$, say. Indeed, this follows from \cite[Corollary 3.20]{Nguyen2020} and the Gauss equation since the surgery scale may be taken arbitrarily small. We can now proceed as in the proof of Proposition \ref{prop:codimension estimate} in the time intervals $(T_i,T_{i+1})$.
\end{proof}

\subsection{The codimension estimate}

Similar reasoning shows that  the codimension estimate holds for surgically modified flows for a suitable range of neck and surgery parameters, and surgery scales.

\begin{proposition}[Codimension estimate for surgically modified flows]\label{prop:codimension estimate with surgery}
Let $\{X_i:M_i^n\times[T_i,T_{i+1}]\to S_K^{n+\ell}\}_{i=1}^{N-1}$, $n\ge 5$, be a surgically modified flow with initial condition in the class $\mathcal{C}_K^{n,\ell}(\alpha,V,\Theta)$ (with $\alpha>\alpha_n$ when $n=5,6,7$). There exist $\delta=\delta(n,\alpha, V, \Theta)>0$ and $C=C(n,\alpha,\Theta)<\infty$ such that
\[
\vert\Ahat\vert^2\le CK^\delta(\vert \mn\vert^2+K)^{1-\delta}\vts \;\;\text{wherever}\;\; \mn\neq 0\vts.
\] 
\end{proposition}
\begin{proof}
As in Proposition \ref{prop:codimension estimate}, we seek a bound for
\[
f_\sigma\doteqdot \begin{cases} \displaystyle \frac{\frac{1}{2}\vert\Ahat\vert^2}{W}W^{\sigma} &\text{if} \;\; \mn\neq 0\\ 0 & \text{if}\;\; \mn =0\end{cases}
\]
for some $\sigma\in(0,1)$. 

The key observation is that, for a suitable range of neck and surgery parameters and surgery scales, $f_\sigma$ is pointwise nonincreasing on regions modified during surgery, provided $\sigma$ is small enough relative to the surgery parameters. This follows from \cite[Corollary 3.20]{Nguyen2020}. In addition, since $\vert\Ahat\vert^2\equiv 0$ on a hypersurface cylinder or cap, we have that $ f_\sigma \equiv 0$ on regions added during surgery. The claim follows since, recalling the proof of Proposition \ref{prop:codimension estimate}, for $\sigma$ small depending on $n$ and $m-\alpha$ we have 
\[
f_\sigma\le \max\left\{\max_{M\times\{T_i\}}f_\sigma,CK^\sigma\right\}\vts.
\]
in each of the time intervals $(T_i,T_{i+1})$, where $C = C(n, m-\alpha)$. 
\end{proof}

\subsection{The cylindrical estimate}

Similar reasoning yields a cylindrical estimate for surgically modified flows.

\begin{proposition}[Cylindrical estimate for surgically modified flows (Cf. {\cite[Theorem 5.3]{HuSi09}})]\label{prop:cylindrical surgery}
Let $\{X_i:M_i^n\times[T_i,T_{i+1}]\to S_K^{n+\ell}\}_{i=1}^{N-1}$, $n\ge 5$, be a surgically modified flow with initial condition in the class $\mathcal{C}_K^{n,\ell}(\alpha,V,\Theta)$ (with $\alpha>\alpha_n$ when $n=5,6,7$). For every $\eta\in(0,\eta_0)$ there exists $C_\eta=C_\eta(n,\ell,\alpha,V,\Theta,\eta)<\infty$ such that
\begin{align}\label{eq:cylindrical estimate surgery}
|\sff|^2-\tfrac{1}{n-1}\vert\mn\vert^2 \leq \eta \vert\mn\vert^2 + C_\eta K \quad\text{in}\quad M_i\times[T_i,T_{i+1}]
\end{align}
for all $i$.
\end{proposition}
\begin{proof}
We can arrange, for suitable neck and surgery parameters, that $(g_{\sigma,\eta})_+$ is pointwise non-increasing in regions modified by surgery, and zero in regions added by surgery. This follows from \cite[Corollary 3.20]{Nguyen2020} and the Gauss equation. Proceeding as in the proof of Proposition \ref{prop:cylindrical estimate}, we obtain analogues of the inequalities \eqref{eq:Stampacchia} and \eqref{eq:v_k monotonicity} on each time interval $(T_i,T_{i+1})$, with $v_k$ replaced by $(g_{\sigma,\eta}-k)_+^{\frac{p}{2}}$. Since $(g_{\sigma,\eta})_+$ is pointwise non-increasing across surgery times, and the total volume of the solution decreases under surgery, the analogue of \eqref{eq:Stampacchia} can be integrated from $T_1=0$ to $T_N=T$ to obtain an analogue of \eqref{eq:still holds with surgeries}. The remainder of the proof of the cylindrical estimate then applies unmodified.
\end{proof}


\subsection{The gradient estimate}

Since the derivatives of the second fundamental form are zero on round Euclidean cylinders and spherical caps, the derivative estimates also pass to surgically modified flows.
 
\begin{proposition}[Gradient estimate for surgically modified flows (cf. {\cite[Theorem 6.1]{HuSi09}})]\label{prop:gradient surgery}
Let $\{X_i:M_i^n\times[T_i,T_{i+1}]\to S_K^{n+\ell}\}_{i=1}^{N-1}$, $n\ge 5$, be a surgically modified flow with initial condition in the class $\mathcal{C}_K^{n,\ell}(\alpha,V,\Theta)$ (with $\alpha>\alpha_n$ when $n=5,6,7$). There exists $C=C(n,\ell,\alpha,V,\Theta)<\infty$ such that the estimate
\begin{align}\label{eq:gradient estimate surgery}
\vert\cd \sff\vert^2\le C(\vert\mn\vert^4+K^2)
\end{align}
holds for all $t \geq \lambda_0 K^{-1}$.  
\end{proposition}
\begin{proof}
We proceed as in the proof of Proposition \ref{prop:gradient estimate}, but with the exponential decay term discarded and fixed $\eta=\beta$. First observe that, since $|\sff|^2\equiv\frac{1}{n-1}\vert\mn\vert^2$ on a hypersurface cylinder and $|\sff|^2\equiv\frac{1}{n}\vert\mn\vert^2$ on a hypersurface cap, we can ensure, for a suitable range of neck and surgery parameters, and surgery scales, that
\begin{align*}
|\sff|^2-\tfrac{1}{n-1}\vert\mn\vert^2 \leq \tfrac{\beta}{2}\vert\mn\vert^2
\end{align*}
on regions modified or added by surgery. We may therefore arrange that
\begin{align*}
G_\beta\doteqdot {}& \left(\tfrac{1}{n-1}+\beta\right)\vert\mn\vert^2-|\sff|^2 + 2C_\beta K\geq \tfrac{\beta}{2}\vert\mn\vert^2 
\end{align*}
and 
\begin{align*}
G_0\doteqdot {}&\tfrac{3}{n+2}\vert\mn\vert^2 -|\sff|^2 + 2C_0 K\geq \tfrac{3}{2}\beta\vert\mn\vert^2.
\end{align*}
Furthermore, since $|\nabla \sff|^2 \equiv 0$ on a hypersurface cylinder or cap, we can ensure, for a suitable range of neck and surgery parameters, and surgery scales, that, on regions modified or added by surgery, $|\nabla \sff|^2 \leq \mu_0\vert\mn\vert^4$, where $\mu_0$ is a constant which depends only on $n$. Thus, in regions modified or added by surgery,
\begin{align*}
\frac{|\nabla \sff|^2}{G_0G_\beta} \leq \frac{4\mu_0}{3\beta^2}\,.
\end{align*} 
Choosing $\beta$ sufficiently small, we may now proceed as in the proof of Proposition \ref{prop:gradient estimate} in the time intervals $(T_i,T_{i+1})$. 
\end{proof}

\subsection{Higher order estimates}

Proceeding similarly as  in Proposition \ref{prop:gradient surgery} (cf. \cite[Theorem 6.3]{HuSi09}) yields estimates for higher derivatives of $\sff$ along surgically modified flows.

\begin{proposition}[Hessian estimate for surgically modified flows (cf. {\cite[Theorem 6.3]{HuSi09}})]\label{prop:Hessian estimate surgery}
Let $\{X_i:M_i^n\times[T_i,T_{i+1}]\to S_K^{n+\ell}\}_{i=1}^{N-1}$, $n\ge 5$, be a surgically modified flow with initial condition in the class $\mathcal{C}_K^{n,\ell}(\alpha,V,\Theta)$ (with $\alpha>\alpha_n$ when $n=5,6,7$). There exists $C=C(n,\ell,\alpha,V,\Theta)$ such that the estimate
\ba\label{eq:scale_invariant_Hessian_estimate_surgery}
\vert\cd^2{\sff}\vert^2\leq C(\vert\mn\vert^{6}+K^{3})\,.
\ea
holds for all $t \geq \lambda_0 K^{-1}$.
\end{proposition}
  
\begin{proof}
Proceed as in Proposition \ref{prop:Hessian estimate} between surgeries and use the fact that, for suitable neck and surgery parameters, and surgery scales, $\vert \cd^2\sff\vert^2/\vert\mn\vert^6$ is small in regions modified or added by surgery.
\end{proof}

\subsection{Neck detection}

The conclusion of the neck detection Lemma \ref{lem:curvature neck detection} also holds for surgically modified flows, so long as we work in regions which are not affected by surgeries (cf. \cite[Lemma 7.4]{HuSi09}). 

In the following theorem, a region $U\times I$ is \emph{free of surgeries} if at each surgery time $T_i\in I$, $i\in \{2,\dots,N-1\}$, we have $U\subset M_{i-1}\cap M_i$ and $X_{i-1}|_U(\vts\cdot\vts,T_{i})=X_{i}|_U(\vts\cdot\vts,T_{i})$ (and hence $X_{i-1}$ and $X_i$ may be pasted together to form a smooth mean curvature flow in $U\times I$).

\begin{lemma}[{Neck detection for surgically modified flows (cf. \cite[Lemma 7.4]{HuSi09})}] \label{lem:NDL} 
Let $\{X_i:M_i^n\times[T_i,T_{i+1}]\to S_K^{n+\ell}\}_{i=1}^{N-1}$, $n\geq 5$, be a surgically modified flow with initial condition in the class $\mathcal C^{n,\ell}_K(\alpha,V,\Theta)$. Given $\varepsilon>0$, $k\ge 2$, $L>0$, and $\theta>0$, there exist $\eta_\sharp=\eta_\sharp(n,\ell,\alpha,V,\Theta,\varepsilon,k,L,\theta)>0$ and $h_\sharp=h_\sharp(n,\ell,\alpha,V,\Theta,\varepsilon,k,L,\theta)<\infty$ with the following property: If
\begin{itemize}
\item[(ND1)] $|\mn(x_0,t_0)| \geq h_\sharp \sqrt K$ and  $\frac{ (|\sff|^2-\frac{1}{n-1}|\mn|^2)(x_0,t_0)}{|\mn(x_0,t_0)|^2}\geq -\eta_\sharp$, and
\item[(ND2)] the neighbourhood $\mathcal{B}_{(L+1)r_0}(x_0,t_0)\times \left(t_0-\theta r_0^2,t_0\right]$ 
is free of surgeries, where $r_0^2\doteqdot \frac{n-1}{|\mn|^2(x_0,t_0)}$,
\end{itemize} 
then $(x_0,t_0)$ lies at the centre of an $(\varepsilon,k,L)$-neck of size $r_0$.
\end{lemma} 
\begin{proof}
The proof of Lemma \ref{lem:curvature neck detection} applies using Propositions \ref{prop:cylindrical surgery}, \ref{prop:gradient surgery} and  \ref{prop:Hessian estimate surgery} in lieu of Propositions \ref{prop:cylindrical estimate}, \ref{prop:gradient estimate} and \ref{prop:Hessian estimate}, due to the hypothesis (ND2).
\end{proof}

Analogous arguments yield versions of the neck detection lemma which establish that a spacetime region of the solution is an evolving neck, even when said region is bordered by part of the solution modified by surgery (see \cite[Lemma 7.4 (i)]{HuSi09} and \cite[Lemma 5.8 (1)]{Nguyen2020}).

\subsection{Hypersurface detection}

Since the codimension estimate survives the surgery, we obtain an analogue of the hypersurface detection lemma in regions unaffected by surgery.

\begin{lemma}[Hypersurface detection for surgically modified flows]\label{lem:HDL}
Let $\{X_i:M_i^n\times[T_i,T_{i+1}]\to S_K^{n+\ell}\}_{i=1}^{N-1}$, $n\geq 5$, be a surgically modified flow in the class $\mathcal C^{n,\ell}_K(\alpha,V,\Theta)$. Given $\varepsilon>0$, there exist $h_\sharp=h_\sharp(n,\ell,\alpha,V,\Theta,\varepsilon)<\infty$, $L_\sharp=L_\sharp(n,\ell,\alpha,V,\Theta)>0$ and $\theta_\sharp=\theta_\sharp(n,\ell,\alpha,V,\Theta)>0$ with the following property. If
\begin{itemize}
\item[(HD1)] $\displaystyle |\mn(x_0, t_0)|\geq h_\sharp\sqrt{K}$, and
\item[(HD2)] the neighbourhood $\mathcal{B}_{(L_\sharp+1)r_0}(x_0,t_0)\times \left(t_0-(\theta_\sharp+1)r_0^2,t_0\right]$ 
is free of surgeries, where $r_0\doteqdot \frac{n-1}{\vert\mn\vert(x_0,t_0)}$,
\end{itemize}
then 
\[
\sup_{\mathcal{B}_{L_\sharp r_0}(x_0,t_0)\times (t_0-\theta_\sharp r_0^2,t_0]}\vert\cd^k\Ahat\vert\le \varepsilon r_0^{-(k+1)}
\]
for each $k=0,\dots,\lfloor\frac{2}{\varepsilon}\rfloor$.
\end{lemma}
\begin{proof}
The proof of Lemma \ref{lem:hypersurface detection} applies using Propositions \ref{prop:cylindrical surgery}, \ref{prop:gradient surgery} and  \ref{prop:Hessian estimate surgery} in lieu of Propositions \ref{prop:cylindrical estimate}, \ref{prop:gradient estimate} and \ref{prop:Hessian estimate}, due to the hypothesis (HD2).
\end{proof}

\section{Existence of terminating surgically modified flows}

We say that a surgically modified flow $\{X_i:M_i^n\times[T_i,T_{i+1}]\to S_K^{n+\ell}\}_{i=1}^{N-1}$ \emph{terminates} at the final time $T\doteqdot T_N<\infty$ if either
\begin{itemize}
\item each connected component of $X_{N-1}(M_{N-1},T_N)$ is diffeomorphic to $S^n$ or to $S^1\times S^{n-1}$, or
\item after performing surgery on $X_{N-1}(M_{N-1},T_N)$, each connected component of the resulting hypersurface is diffeomorphic to $S^n$ or to $S^1\times S^{n-1}$.
\end{itemize}

\begin{theorem}[Existence of terminating surgically modified flows]
Let $X:M\to S_K^{n+\ell}$, $n\ge 5$, be a properly immersed hypersurface satisfying the quadratic pinching condition \eqref{eq:strict quadratic pinching}. There exists a surgically modified flow $\{X_i:M_i^n\times[T_i,T_{i+1}]\to S_K^{n+\ell}\}_{i=1}^{N-1}$ with $X_1(\vts\cdot\vts,0)=X$ which terminates at time $T=T_N$.
\end{theorem}
\begin{proof}
Given the codimension, cylindrical and derivative estimates, and the neck and hypersurface detection lemmas, and a sufficiently small choice of the surgery scale $r$, we can proceed as in \cite[Section 6]{Nguyen2020} (cf. \cite[Section 8]{HuSi09}) using the machinery developed in \cite[Section 3]{Nguyen2020} (cf. \cite[Sections 3 and 7]{HuSi09}), with only minor modifications required (cf. \cite{LangfordNguyen}). These are:

1. In order to reconcile our data $\mathcal{C}_K^n(\alpha,V,\Theta)$ with those of \cite{Nguyen2020}, we replace the parameter $K$ by introducing the scale factor $R\doteqdot 1/\sqrt{\Theta K}$. Our data $\alpha$ and $V$ can then be related to the $\alpha_0$ and $\alpha_2$ there, respectively. The constant $\alpha_1$ which appears in \cite{Nguyen2020} is not needed here. Since the surgery scale may be taken as small as needed, we may then choose the surgery parameters (albeit with slightly worse values) as explained in \cite[Section 6]{Nguyen2020}.

2. Since our ambient space is non-Euclidean, the proof of the neck continuation theorem \cite[Theorem 6.3]{Nguyen2020} requires modification in two places. These are explained and carried out in detail in \cite[Section 8]{BrHu17}. 

3. Since the maximal time is not a priori bounded in the present setting, the surgery algorithm may not terminate ``on its own''. This case is easily dealt with using 
Proposition \ref{prop:cylindrical decay} as in \cite{LangfordNguyen}, however.
\end{proof}

Theorem \ref{thm:main theorem} follows.


\begin{thebibliography}{10}

\bibitem{AlencarDoCarmo}
Hil\'{a}rio Alencar and Manfredo do~Carmo.
\newblock Hypersurfaces with constant mean curvature in spheres.
\newblock {\em Proc. Amer. Math. Soc.}, 120(4):1223--1229, 1994.

\bibitem{LiLi}
Li~An-Min and Li~Jimin.
\newblock An intrinsic rigidity theorem for minimal submanifolds in a sphere.
\newblock {\em Arch. Math. (Basel)}, 58(6):582--594, 1992.

\bibitem{An02}
B.~Andrews.
\newblock Positively curved surfaces in the three-sphere.
\newblock In {\em Proceedings of the {I}nternational {C}ongress of
  {M}athematicians, {V}ol. {II} ({B}eijing, 2002)}, pages 221--230. Higher Ed.
  Press, Beijing, 2002.

\bibitem{AnBa10}
Ben Andrews and Charles Baker.
\newblock Mean curvature flow of pinched submanifolds to spheres.
\newblock {\em J. Differential Geom.}, 85(3):357--395, 2010.

\bibitem{EGF}
Ben Andrews, Bennett Chow, Christine Guenther, and Mat Langford.
\newblock {\em {E}xtrinsic {G}eometric {F}lows}.
\newblock {G}raduate {S}tudies in {M}athematics, Vol. 206. {A}merican
  {M}athematical {S}ociety, 2020.

\bibitem{BakerThesis}
Charles Baker.
\newblock {\em The mean curvature flow of submanifolds of high codimension}.
\newblock Australian National University, 2011.
\newblock Thesis (Ph.D.)--Australian National University.

\bibitem{BakerNguyen2ElectricBoogaloo}
Charles Baker and Huy~The Nguyen.
\newblock Evolving pinched submanifolds of the sphere by mean curvature flow.
\newblock Preprint, \href{https://arxiv.org/abs/2004.12259}{arXiv:2004.12259}.

\bibitem{BakerNguyen17}
Charles Baker and Huy~The Nguyen.
\newblock Codimension two surfaces pinched by normal curvature evolving by mean
  curvature flow.
\newblock {\em Ann. Inst. H. Poincar\'{e} Anal. Non Lin\'{e}aire},
  34(6):1599--1610, 2017.

\bibitem{BrHu17}
Simon Brendle and Gerhard Huisken.
\newblock A fully nonlinear flow for two-convex hypersurfaces in {R}iemannian
  manifolds.
\newblock {\em Invent. Math.}, 210(2):559--613, 2017.

\bibitem{ChengNakagawa}
Qing~Ming Cheng and Hisao Nakagawa.
\newblock Totally umbilic hypersurfaces.
\newblock {\em Hiroshima Math. J.}, 20(1):1--10, 1990.

\bibitem{ChernDoCarmoKobayashi}
S.~S. Chern, M.~do~Carmo, and S.~Kobayashi.
\newblock Minimal submanifolds of a sphere with second fundamental form of
  constant length.
\newblock In {\em Functional {A}nalysis and {R}elated {F}ields ({P}roc. {C}onf.
  for {M}. {S}tone, {U}niv. {C}hicago, {C}hicago, {I}ll., 1968)}, pages 59--75.
  Springer, New York, 1970.

\bibitem{HoSp74}
David {Hoffman} and Joel {Spruck}.
\newblock {Sobolev and isoperimetric inequalities for Riemannian submanifolds.}
\newblock {\em {Commun. Pure Appl. Math.}}, 27:715--727, 1974.

\bibitem{Hu84}
Gerhard Huisken.
\newblock Flow by mean curvature of convex surfaces into spheres.
\newblock {\em J. Differential Geom.}, 20(1):237--266, 1984.

\bibitem{Hu87}
Gerhard Huisken.
\newblock Deforming hypersurfaces of the sphere by their mean curvature.
\newblock {\em Math. Z.}, 195(2):205--219, 1987.

\bibitem{HuSi09}
Gerhard Huisken and Carlo Sinestrari.
\newblock Mean curvature flow with surgeries of two-convex hypersurfaces.
\newblock {\em Invent. Math.}, 175(1):137--221, 2009.

\bibitem{LangfordNguyen}
Mat Langford and Huy~The Nguyen.
\newblock Quadratically pinched hypersurfaces of the sphere via mean curvature
  flow with surgery.
\newblock {\em Calc. Var. Partial Differ. Equ.}, 60(6):1-33, 2021. 

\bibitem{LaNgPinching}
Mat Langford and Huy~The Nguyen.
\newblock Sharp pinching estimates for mean curvature flow in the sphere.
\newblock \href{https://arxiv.org/abs/2009.00986}{arXiv:2009.00986}, 2020.

\bibitem{LyNgConvexity}
Stephen Lynch and Huy~The Nguyen.
\newblock Convexity estimates for high codimension mean curvature flow.
\newblock Preprint, \href{https://arxiv.org/abs/2006.05227}{arXiv:2006.05227
  [math.DG]}.

\bibitem{LyNg}
Stephen Lynch and Huy~The Nguyen.
\newblock Pinched ancient solutions to the high codimension mean curvature
  flow.
\newblock {\em Calc. Var. Partial Differ. Equ.}, 60(1):1-14, 2021.

\bibitem{MiSi73}
J.~H. Michael and L.~M. Simon.
\newblock Sobolev and mean-value inequalities on generalized submanifolds of
  {$R^{n}$}.
\newblock {\em Comm. Pure Appl. Math.}, 26:361--379, 1973.

\bibitem{Naff}
Keaton Naff.
\newblock Codimension estimates in mean curvature flow.
\newblock Preprint, \href{https://arxiv.org/abs/1906.08184}{arXiv:1906.08184}.

\bibitem{Nguyen2020}
Huy~The Nguyen.
\newblock High codimension mean curvature flow with surgery.
\newblock \href{https://arxiv.org/abs/2004.07163}{arXiv:2004.07163}, 2020.

\bibitem{Okumura}
Masafumi Okumura.
\newblock Hypersurfaces and a pinching problem on the second fundamental
  tensor.
\newblock {\em Amer. J. Math.}, 96:207--213, 1974.

\bibitem{Risa}
Susanna Risa and Carlo Sinestrari.
\newblock Ancient solutions of geometric flows with curvature pinching.
\newblock {\em J. Geom. Anal.}, 29(2):1206--1232, 2019.

\bibitem{Santos}
Walcy Santos.
\newblock Submanifolds with parallel mean curvature vector in spheres.
\newblock {\em Tohoku Math. J. (2)}, 46(3):403--415, 1994.

\bibitem{Si68}
James Simons.
\newblock Minimal varieties in riemannian manifolds.
\newblock {\em Ann. of Math. (2)}, 88:62--105, 1968.

\bibitem{St66}
Guido Stampacchia.
\newblock {\em \`{E}quations elliptiques du second ordre \`a coefficients
  discontinus}.
\newblock S\'eminaire de Math\'ematiques Sup\'erieures, No. 16 (\'Et\'e, 1965).
  Les Presses de l'Universit\'e de Montr\'eal, Montreal, Que., 1966.

\end{thebibliography}
\end{document}